\documentclass{article}

\usepackage{amsmath,amssymb,amsfonts,amsthm}
\usepackage{color,graphicx,bbm,verbatim}
\usepackage{mathtools}
\usepackage{epstopdf}
\usepackage{subfig}
\usepackage{arydshln}
\usepackage{booktabs}
\usepackage{hyperref}
\usepackage[capitalize]{cleveref}
\usepackage[a4paper,tmargin=1.5in,bmargin=1.5in]{geometry}
\usepackage[normalem]{ulem}
\usepackage{sectsty}

\sectionfont{\fontsize{12}{15}\selectfont}

\newtheorem{theorem}{Theorem}[section] 
\newtheorem{definition}[theorem]{Definition}
\newtheorem{lemma}[theorem]{Lemma}
\newtheorem{corollary}[theorem]{Corollary}
\newtheorem{remark}[theorem]{Remark}
\newtheorem{keyremark}[theorem]{Key Remark}

\numberwithin{equation}{section}

\usepackage{amsopn}

\ifpdf
\hypersetup{
  pdftitle={Computing the Kreiss Constant of a Matrix},
  pdfauthor={Tim Mitchell}
}
\fi

\def\matlab{MATLAB}

\def\R{\mathbb{R}}
\def\C{\mathbb{C}}
\def\N{\mathbb{N}}
\def\bigO{\mathcal{O}}
\def\imagunit{\mathbf{i}}
\def\transsym{\mathsf{T}}
\newcommand{\tp}[1][]{^{{#1}\transsym}}

\DeclareMathOperator{\argmin}{\arg\min}

\DeclareMathOperator{\vecop}{\mathrm{vec}}
\DeclareMathOperator{\bd}{\mathrm{bd}}

\def\e{\mathrm{e}}
\def\eit{\eix{\theta}}
\def\emit{\emix{\theta}}

\newcommand{\eix}[1]{\e^{\imagunit #1}}
\newcommand{\emix}[1]{\e^{-\imagunit #1}}

\def\eps{\varepsilon}
\def\smin{\sigma_{\min}}

\def\aleps{\alpha_\eps}
\def\rhoeps{\rho_\eps}
\def\dtu{\tau(A,B)}

\def\Kcal{\mathcal{K}}
\def\Kcon{\Kcal(A)}
\def\Kinv{\Kcon^{-1}}
\def\n{n}

\def\costh{\cos \theta}
\def\sinth{\sin \theta}
\def\shift{s}

\renewcommand{\Re}{\mathrm{Re}\,}

\def\beq{\begin{equation}}
\def\eeq{\end{equation}}
\def\beqs{\begin{equation*}}
\def\eeqs{\end{equation*}}
\def\bseq{\begin{subequations}}
\def\eseq{\end{subequations}}

\usepackage{caption}
\usepackage{newfloat}
\usepackage[section]{algorithm}
\usepackage{algorithmic}
\DeclareFloatingEnvironment[placement=htb]{algfloat}
\newcommand{\algnote}[1]{\footnotesize \sc{Note: \it#1 } }

\newcommand{\algorithmicbreak}{\textbf{break}}
\newcommand{\BREAK}{\STATE \algorithmicbreak}

\title{Computing the Kreiss Constant of a Matrix\thanks{The author's visits to the 
Courant Institute, NYU,
where some of this work was conducted, 
were supported by the U.S. National Science Foundation grant DMS-1620083.}}
\author{
Tim Mitchell\thanks{
Max Planck Institute for Dynamics of Complex Technical Systems, Magdeburg, 39106 Germany \href{mailto:mitchell@mpi-magdeburg.mpg.de}{\texttt{mitchell@mpi-magdeburg.mpg.de}}.}
}
\date{\small{July 15th, 2019\\ Revised: May 20, 2020, August 15, 2020}}

\begin{document}

\maketitle

\begin{abstract}
We establish the first globally convergent algorithms for computing
the Kreiss constant of a matrix to arbitrary accuracy.
We propose three different iterations for continuous-time Kreiss constants
and analogues for discrete-time Kreiss constants.
With standard eigensolvers, the methods do $\bigO(\n^6)$ work,
but we show how this theoretical work complexity can be lowered
to $\bigO(\n^4)$ on average and $\bigO(\n^5)$ in the worst case
via divide-and-conquer variants.
Finally, locally optimal Kreiss constant approximations
can be efficiently obtained for large-scale matrices via optimization.
\end{abstract}


\noindent
\textbf{Notation:} $\| \cdot \|$ denotes the spectral norm, $\smin(\cdot)$ the smallest singular value, 
$\Lambda(\cdot)$ the spectrum, 
$J = \begin{bsmallmatrix} 0 & I \\ -I & 0 \end{bsmallmatrix}$, a matrix $A \in \C^{2\n\times 2\n}$
is Hamiltonian if $(JA)^* = JA$ and symplectic if $A^*JA = J$,
a matrix pencil $A - \lambda B$ is symplectic if $A^*JA = B^*JB$,
$\mathrm{e}$~is~Euler's~number~$2.71828\ldots$, and $I_n$ is the $\n \times \n$ identity,
though we will often omit the subscript when the dimension is clear.

\section{Introduction}
Given a matrix $A\in\C^{\n\times \n}$, 
the ordinary difference equation
\beq
	\label{eq:ode_disc}
	x_{k+1} = Ax_k
\eeq
is asymptotically stable if $A$ is Schur stable, i.e., if $\rho(A) < 1$, where 
$\rho$ denotes the spectral radius.
While $\rho(A)$ tells one about the asymptotic behavior 
of \eqref{eq:ode_disc}, it does not convey information about its transient behavior.
For that, we can look at the Kreiss Matrix Theorem, 
which says for any matrix $A \in\C^{\n\times \n}$ \cite[Eq.~18.2]{TreE05}
\beq
	\label{eq:kreiss_thm_disc}
	\Kcon \leq \sup_{k \geq 0} \|A^k\| \leq \e\n\Kcon,
\eeq
where the \emph{Kreiss constant} $\Kcon$ is given by \cite[p.~143]{TreE05}
\beq	
	\label{eq:k2d_disc}
	\Kcon = \sup_{z \in \C, |z| > 1} 
	(|z|  - 1) \| (z I - A)^{-1} \|.
\eeq
As also noted in \cite[p.~143]{TreE05}, $\Kcon$ has an equivalent formulation of
\beq
	\label{eq:k1d_disc}
	\Kcon = \sup_{\eps > 0} \frac{\rhoeps(A) - 1}{\eps},
\eeq
where the \emph{$\eps$-pseudospectral radius} $\rhoeps$ is defined by
\bseq
\begin{align}
	\rhoeps(A) 	&= \max \{ |z|  : z \in \Lambda(A+\Delta), \|\Delta\| \leq \eps \} \\
				&= \max \{ |z|  : z \in \C, \| (zI - A)^{-1} \| \geq \eps^{-1} \}.
\end{align}
\eseq
For any matrix $A \in \C^{\n \times \n}$, $\Kcon \geq 1$, and $\Kcon$ may be arbitrarily large.
As is well known, a matrix $A$ is \emph{power-bounded}, i.e., $\Kcon < \infty$, if and only if
$\rho(A) \leq 1$ and all eigenvalues of $A$ with modulus 1 are nondefective.
If $A$ is normal and $\rho(A) \leq 1$, then $\Kcon = 1$.

As discussed on \cite[p.~177]{TreE05}, 
the original statement by Kreiss in 1962 \cite{Kre62}
actually had a far looser upper bound than \eqref{eq:kreiss_thm_disc}: approximately $c^{\n^\n} \Kcon$.
The reduction of the constant factor to its current form in fact occurred over nearly thirty years in at least nine separate steps,
with Spijker proving the conjecture of \cite[p.~590]{LevT84} to finally obtain the (in a certain sense) tight factor of $\e\n$ in 1991 \cite{Spi91}.

The Kreiss Matrix Theorem also comes in a continuous-time variant 
for an ordinary differential equation
\beq
	\label{eq:ode_cont}
	\dot x = Ax,
\eeq
which is asymptotically stable if $A$ is Hurwitz stable, i.e., if $\alpha(A) < 0$, where 
$\alpha$ denotes the spectral abscissa.
In this case, the Kreiss Matrix Theorem states \cite[Eq.~18.8]{TreE05}
\beq
	\label{eq:kreiss_thm_cont}
	\Kcon \leq \sup_{t \geq 0} \|\e^{tA}\| \leq \e\n\Kcon
\eeq
where by \cite[Eq.~14.7]{TreE05}, $\Kcon$ is now equivalently given by either
\beq	
	\label{eq:k2d_cont}
	\Kcon = \sup_{z \in \C, \Re z > 0} 
	(\Re z) \| (z I - A)^{-1} \|
\eeq
or
\beq
	\label{eq:k1d_cont}
	\Kcon = \sup_{\eps > 0} \frac{\aleps(A)}{\eps},
\eeq
and where the \emph{$\eps$-pseudospectral abscissa} $\aleps$ is defined by
\bseq
\begin{align}
	\aleps(A) 	&= \max \{ \Re z  : z \in \Lambda(A+\Delta), \|\Delta\| \leq \eps \} \\
				&= \max \{ \Re z  : z \in \C, \| (zI - A)^{-1} \| \geq \eps^{-1} \}.
\end{align}
\eseq
Like the discrete-time case, $\Kcon \geq 1$ and can be arbitrary large.
If $A$ is normal and $\alpha(A) \leq 0$, then $\Kcon = 1$.

Despite the wealth of work done over decades towards making the 
upper bound of the Kreiss Matrix Theorem now tight, 
there has been no algorithm given to actually compute $\Kcon$ with guarantees.
In the literature, $\Kcon$ is often just approximated by 
plotting \eqref{eq:k1d_disc} or \eqref{eq:k1d_cont}; 
e.g., see \cite{EmbK17} and \cite[Chapter~3.4.1]{Men06}.

In this paper, we propose the first globally convergent algorithms to compute both 
continuous- and discrete-time Kreiss constants to arbitrary accuracy.
We assume that $A$ is nonnormal, 
as otherwise computing its Kreiss constant just involves checking if $A$ is unstable.
Furthermore, we assume that $\alpha(A) < 0$ or $\rho(A) < 1$ holds, respectively, 
in the continuous- or discrete-time case.  
With standard eigensolvers, the three different methods we propose all have $\bigO(\n^6)$
work complexities,\footnote{
We use the standard convention of treating dense computations 
of singular values, eigenvalues, solutions of Sylvester equations, etc., as \emph{atomic operations} 
with cubic costs in the dimensions of the associated matrices.
We additionally assume these costs become linear in the dimension of the matrices
when corresponding sparse methods are available.
}
but by also developing so-called \emph{divide-and-conquer} variants,
we show how this theoretical work complexity reduces
to  $\bigO(\n^4)$ on average and $\bigO(\n^5)$ in the worst case.
Our work also shows that locally optimal approximations to $\Kcon$
can be efficiently and reliably obtained for large-scale matrices via standard optimization
techniques.
Furthermore, we establish some variational properties of Kreiss constants, 
including that the Kreiss constant of a matrix $A$ is not a continuous function with respect to the entries 
of $A$.
Finally, as a side effect of our work, we also propose an important modification 
to the \emph{distance-to-uncontrollability} algorithms of \cite{Gu00,BurLO04,GuMOetal06} 
to greatly improve their reliability in practice.

The contributions of the paper are structured as follows.
In \S\ref{sec:dtu}, we introduce the so-called distance to uncontrollability and
present a theorem of Gu \cite[Theorem~3.1]{Gu00} and 
distance-to-uncontrollability algorithms that are based upon it.
In \S\ref{sec:varprop}, we establish variational properties of Kreiss constants that we will need
and show that there is a potentially exploitable similarity to computing the distance to uncontrollability.
Then, in \S\ref{sec:analogues}, we develop theorems for continuous-time $\Kcon$ that are 
analogues of the aforementioned theorem of Gu,
but nevertheless show that, due to key structural differences, 
existing distance-to-uncontrollability algorithms will not directly extend to Kreiss constants.
By developing a so-called \emph{globality certificate} in \S\ref{sec:cont},
we present our first algorithm for computing continuous-time Kreiss constants,
which is an optimization-with-restarts method \emph{using backtracking},
and then establish an asymptotically faster \emph{divide-and-conquer} variant that is inspired by \cite{GuMOetal06}.
We then modify the premise of our globality certificate in a crucial way in \S\ref{sec:alternatives}
to develop a second certificate with significant structural differences and properties;
this alternative certificate enables two other algorithms for continuous-time $\Kcon$, 
an optimization-with-restarts method \emph{without backtracking} and a trisection iteration,
which can be considered closer analogues of the distance-to-uncontrollability algorithms of \cite{BurLO04}.
Faster divide-and-conquer versions of these two methods are also developed.
In \S\ref{sec:disc}, we consider the case of discrete-time $\Kcon$ and
present discrete-time analogues of all of our continuous-time algorithms and associated theoretical results;
to the best of our knowledge, this is also the first extension of the 2D level-set ideas of \cite{Gu00}
to a discrete-time setting, and it turns out to have surprising differences.
Numerical examples are presented in \S\ref{sec:experiments}, with concluding
remarks made in \S\ref{sec:conclusion}.

For this paper, the following general definition
and two theorems will be needed.  The theorems can be found in several places in various forms, such as
\cite{Lan64,OveW95} 
and \cite[Theorem~13.16]{Lau05}, respectively.

\begin{definition}
Given a domain $\Omega \subseteq \R$, a function $f: \Omega \to \R$ 
has a \emph{global Lipschitz constant (GLC)} of $c \geq 0$
if $|f(x) - f(y) | \leq c |x - y|$ for all $x,y \in \Omega$.
\end{definition}

\begin{theorem}
\label{thm:eig2ndderiv}
For $x,y \in \R$, let $A(x,y)$ be a twice-differentiable $\n \times \n$ Hermitian matrix family,
and for a point $(\hat x,\hat y)$, 
let $\lambda_1 \geq \ldots \geq \lambda_\n$ be the eigenvalues of $A(\hat x,\hat y)$ with associated unit-norm eigenvectors $q_1,\ldots,q_\n$.  Then assuming $\lambda_j$ is unique,
\[
	\tfrac{\partial^2}{\partial x \partial y} \lambda_j \bigg|_{x=\hat x, y = \hat y} = 
	q_j^* \tfrac{\partial^2 A(\hat x,\hat y)}{\partial x \partial y}  q_j 
	+ 2 \sum_{k \ne j}
	\frac{ 
		q_j^* \tfrac{\partial A(\hat x, \hat y)}{\partial x} q_k  
		\cdot 
		q_j^* \tfrac{\partial A(\hat x, \hat y)}{\partial y} q_k
	}{
		\lambda_j - \lambda_k
	}.
\]
\end{theorem}

\begin{theorem}
\label{thm:kronsum}
Let $A \in \C^{n \times n}$ with $\lambda_j \in \Lambda(A)$ for $j = 1,\ldots,n$ and 
\mbox{$B \in \C^{m \times m}$} with $\mu_k \in \Lambda(B)$ for $k = 1,\ldots,m$.
Then the \emph{Kronecker sum} $A \oplus B = I_m \otimes A + B \otimes I_n$
has eigenvalues $\lambda_j + \mu_k$, for all pairs of $j$ and $k$.
\end{theorem}

\section{Computing the distance to uncontrollability}
\label{sec:dtu}
Given $A \in \C^{\n \times \n}$ and $B \in \C^{\n \times m}$, consider
the linear control system
\beq
	\label{eq:xAB}
	\dot x = Ax + Bu,
\eeq
where the state $x \in \C^\n$ and $u \in \C^m$, the \emph{control input}, are both dependent
on time.
The system \eqref{eq:xAB} is \emph{controllable} if given respective initial and final states $x(0)$ and $x(T)$ with $T > 0$,
there exists a control $u(\cdot)$ that realizes some trajectory $x(\cdot)$ with endpoints  $x(0)$ and $x(T)$.  The \emph{distance to uncontrollability}, which we denote as $\dtu$, can be computed 
via solving the nonconvex optimization problem \cite{Eis84}
\beq
	\label{eq:dtu}
	\dtu = \min_{z \in \C} \smin\left(\begin{bmatrix} A - zI & B \end{bmatrix}\right)
	= \min_{x,y \in \R} f(x,y),
\eeq
where $f(x,y) = \smin(F(x,y))$ and 
$F(x,y) = \begin{bmatrix} A - (x + \imagunit y)I & B \end{bmatrix}$.
The first practical algorithm to address computing $\dtu$ is due to Gu \cite{Gu00},
based on the following result \cite[Theorem~3.1]{Gu00}.

\begin{theorem}
\label{thm:dtu}
Let $\gamma, \eta \geq 0$ be given.
If $\dtu \leq \gamma$ and $\eta \in [0,2(\gamma - \dtu)]$, then 
there exists a pair $x,y \in \R$ such that
\beq
	\label{eq:dtu_ub}
	f(x,y) = f(x+\eta,y) = \gamma.
\eeq
\end{theorem}

\begin{corollary}
\label{cor:dtu}
Let $\gamma, \eta \geq 0$ be given.
If there do not exist any pairs $x,y \in \R$ such that \eqref{eq:dtu_ub} holds,
then 
\beq
	\label{eq:dtu_lb}
	\dtu > \gamma - \tfrac{\eta}{2}.
\eeq
\end{corollary}
The proof of \cref{thm:dtu} relies on the fact that $f(x,y)$ 
has a GLC of 1 with respect to either $x$ or $y$.

What \cite[Section~3.2]{Gu00} additionally devised was a sequence of computations 
to verify whether either \eqref{eq:dtu_ub} or \eqref{eq:dtu_lb} holds for a given choice of $\gamma$ and $\eta$.
This verification procedure, \emph{when using exact arithmetic}, is able to detect and find any points 
$(x,y)$ such that \eqref{eq:dtu_ub} is satisfied.
If so, the test returns these points
and $\dtu \leq \gamma$ is verified.  Otherwise,
the test asserts no pairs satisfy \eqref{eq:dtu_ub} and so \eqref{eq:dtu_lb} instead must hold.
Using this procedure, Gu proposed a bisection-like scheme to estimate $\dtu$ to \emph{within a factor of two}.
For initialization, $\gamma \coloneqq f(0,0)$ and $\eta \coloneqq \gamma$.
If the test verifies $\dtu \leq \gamma$, then 
$\gamma$ and $\eta$ are both halved (so $\eta = \gamma$ still holds) and the test is done with these smaller values.
Otherwise, \eqref{eq:dtu_lb} holds and so 
$\gamma$ and $\dtu$ are within a factor of two and Gu's method terminates.

As noted in \cite[p.~358]{BurLO04}, it is tempting to try to obtain $\dtu$ to higher precision via a true bisection method, i.e., one that would update both upper and lower bounds, unlike Gu's method which only 
updates an upper bound. 
The problem with this approach is that in order to ascertain whether the current estimate $\gamma$ is
essentially a lower bound to $\dtu$ via \eqref{eq:dtu_lb}, one would have to perform the verification
procedure for $\eta \approx 0$.  
Unfortunately, this is not tenable in the presence of rounding errors, as Gu's procedure becomes more and more 
numerically unreliable as $\eta \to 0$, i.e.,
points satisfying \eqref{eq:dtu_ub} may not be detected.
Consequently, in practice, the lower bound will generally be erroneously updated at some point,
thus preventing convergence to $\dtu$.\footnote{In \cref{rem:improved_test}, 
we discuss the unreliability of Gu's procedure in more detail and
explain how our new tests avoid key numerical pitfalls.
Besides being useful for our $\Kcon$ algorithms, our modifications can also
improve the reliability of the $\dtu$ methods of \cite{Gu00,BurLO04,GuMOetal06}.}

Using Gu's verification procedure (as Gu specified it, i.e., without modifications), 
Burke, Lewis, and Overton instead proposed a trisection algorithm \cite[Algorithm~5.2]{BurLO04} 
that balances how much the lower bound is updated with how quickly the value of $\eta$ is decreased,
precisely to postpone the numerical unreliability of Gu's procedure as long as possible.
Trisection works as follows.
Let $L\coloneqq0$ and $U \coloneqq f(0,0)$ be initial lower and upper bounds, respectively.
Then on the $k$th iteration, $\eta_k \coloneqq \tfrac{2}{3}(U - L)$ and $\gamma_k \coloneqq L + \eta_k$
are set as the current values of $\eta$ and $\gamma$ for the verification test.  
If the test finds points satisfying \eqref{eq:dtu_ub},
then the upper bound is updated $U \coloneqq \gamma_k$.  Otherwise, \eqref{eq:dtu_lb} holds
so we know that $\dtu \geq \gamma_k - \tfrac{\eta_k}{2} = L + \tfrac{\eta_k}{2}$,
and so now the lower bound can be updated $L \coloneqq L + \tfrac{\eta_k}{2}$.
Thus, the new interval has length $\tfrac{2}{3}(U - L)$.
As the trisection algorithm is linearly convergent, 
Gu's verification procedure will have to be invoked many times, 
which already is $\bigO(\n^6)$ work by itself using standard dense eigensolvers.  
Thus, the trisection algorithm also has a large constant factor term hidden away in its asymptotic work complexity.  
Also, trisection is not a panacea for the numerical issues of Gu's procedure; 
although $\eta_k \to 0$ only in the limit, if $\dtu$ is small, then $\eta_k$
must become commensurately small in order for trisection to attain any digits of accuracy;
see \cref{lem:tri_eta} and \cref{cor:rel_acc} in \cref{apdx:lemmas}.

In the same paper, Burke, Lewis, and Overton also proposed a second algorithm for $\dtu$
\cite[Algorithm~5.3]{BurLO04} and advocated it as preferable to trisection.
This \emph{optimization-with-restarts} method also relies on Gu's verification procedure,
now as a \emph{globality certificate}, and additionally, 
on the fact that $f(x,y)$ is semialgebraic and so $f(x,y)$ only has a finite number of 
locally minimal values; see \cite[p.~359]{BurLO04}.
This second method thus works by using optimization techniques 
to find a minimizer of \eqref{eq:dtu} with function value $f_k$ and then uses Gu's verification procedure with 
carefully chosen values of $\gamma$ and $\eta$ so that the test checks if $f_k$ is 
sufficiently close to $\dtu$; for some relative tolerance $\texttt{tol} > 0$, the specific values are
$\gamma \coloneqq f_k \cdot (1 - 0.5\cdot \texttt{tol})$ and $\eta \coloneqq f_k \cdot \texttt{tol}$.\footnote{
Note that \cite{BurLO04} writes these in an equivalent but different form using  
$\delta_1 = \gamma$ and $\delta_2 = \gamma - \tfrac{\eta}{2}$.}
Otherwise, if $f_k \not\approx \dtu$ to tolerance, the certificate provides one or more new starting points from which optimization can be restarted with the guarantee that a better (lower) minimum of \eqref{eq:dtu} will be found, hence optimization is restarted in a loop until the certificate indeed asserts that the desired accuracy has been attained.
By construction, $f_k$ is monotonically decreasing and optimization-with-restarts must terminate with $f_k \approx \dtu$ to tolerance in a finite number of restarts.
Although it is not clear exactly how many restarts will occur, 
only a handful are typically needed in practice, if any.
Furthermore, the optimization phases are relatively cheap, requiring $\bigO(\n^3)$ work with a relatively low constant factor, since minimizers of $f(x,y)$ can generally be found with superlinear or even quadratic convergence.
As a result, optimization-with-restarts is almost always many times faster than trisection.
However, optimization-with-restarts can still be susceptible to 
the numerical difficulties of Gu's verification procedure, 
since $\eta$ itself may still become very small, e.g., 
if either high accuracy is desired or $\dtu$ is small.

Finally, to address the high cost of Gu's verification procedure, though not necessarily its numerical issues,
\cite{GuMOetal06} proposed a divide-and-conquer strategy that lowers the asymptotic work complexity
of Gu's procedure to $\bigO(\n^4)$ on average and $\bigO(\n^5)$ in the worst case.
This benefits all of the aforementioned algorithms.

\section{Variational properties and the inverse of the Kreiss constant}
\label{sec:varprop}
We now establish some variational properties of Kreiss constants, which in turn show that locally optimal approximations to $\Kcon$ can be efficiently computed via optimization, even if $A$ is large.
We also show how the problem of computing $\Kcon$ shares some similarity with computing $\dtu$. 
We begin with the following result.

\begin{lemma}
The Kreiss constant $\mathcal{K}$ is not always continuous at $A$, as it may instantaneously jump to/from $\infty$.
However, $\mathcal{K}$ is continuous at $A$ 
if $\alpha(A) < 0$ holds (continuous-time case)
or $\rho(A) < 1$ holds (discrete-time case).
\end{lemma}
\begin{proof}
We begin with the second part.  If $\alpha(A) < 0$, then $zI - A$ is invertible 
for all $z$ such that $\Re z > 0$.
By continuity of singular values, 
$\|(zI - A)^{-1}\|$ in \eqref{eq:k2d_cont} is continuous at $A$,
and thus so is $\Kcon$.
Via an analogous argument with \eqref{eq:k2d_disc}, 
the continuity claim also holds for the discrete-time case.

We prove the first part by example.  
Let $A(\delta) \coloneqq \begin{bsmallmatrix} -0.5 & 0 \\ 0 & \delta \end{bsmallmatrix}$ for real scalar $\delta \geq 0$.
As $A(\delta)$ is always normal and $\alpha(A(0)) = 0$, $\mathcal{K}(A(0)) = 1$ holds.
However, as $\alpha(A(\delta)) > 0$ for any $\delta > 0$, $\mathcal{K}(A(\delta)) = \infty$ for any $\delta > 0$.
Using the same example with $\delta \geq 1$ for the discrete-time case, we have that $\mathcal{K}(A(1)) = 1$
and $\mathcal{K}(A(\delta)) = \infty$ for all $\delta > 1$.
\end{proof}

\subsection{The continuous-time case}
\label{sec:opt_cont}
Identifying $\C$ with $\R^2$, consider the inverse of the continuous-time Kreiss constant \eqref{eq:k2d_cont}, i.e.,
\beq	
	\label{eq:kinv_cont}
	\Kinv = \inf_{x > 0, y \in \R} 
	\smin \left(\frac{(x+\imagunit y) I - A}{x} \right)
	= \inf_{x > 0, y \in \R} g(x,y),
\eeq
where
\beq
	\label{eq:g_cont}
	g(x,y) = \smin(G(x,y)) \quad \text{and} \quad G(x,y) = \frac{(x+\imagunit y) I - A}{x}.
\eeq
Like $f(x,y)$, $g(x,y)$ is semialgebraic, 
and so in the open right half-plane, $g(x,y)$ must have only a finite number of locally minimal function values.
We now derive the gradient and Hessian of $g(x,y)$ which will be useful for 
finding minimizers via quasi-Newton or Newton methods.
Although singular values can vary nonsmoothly with respect to matrix entries,
they are nevertheless locally Lipschitz, and so this nonsmoothness is confined 
to a set of measure zero.
We first need the first partial derivatives of $G(x,y)$ for $x \ne 0$:
\beq
	\label{eq:first_partials_cont}
	\frac{\partial G(x,y)}{\partial x}  = \frac{xI - ((x + \imagunit y) I - A)}{x^2} = \frac{A - \imagunit y I }{x^2}
	\quad \text{and} \quad
	\frac{\partial G(x,y)}{\partial y} = \frac{\imagunit I}{x}.
\eeq
Let $(\hat x,\hat y)$ be such that $g(\hat x,\hat y) \neq 0$ is a simple singular value of $G(x,y)$ 
with associated left and right singular vectors $u$ and $v$.
Then, by standard perturbation theory for singular values, it follows that
\beq
	\nabla g(\hat x,\hat y) 
	= \Re
	\begin{bmatrix}
	u^* \tfrac{\partial G(x,y)}{\partial x} v \\[6pt]
	u^* \tfrac{\partial G(x,y)}{\partial y} v
	\end{bmatrix}.
\eeq
Now since $g(x,y) = \smin(G(x,y))$ is also the $\n$th eigenvalue (in descending order) of the $2\n \times 2\n$ Hermitian matrix 
\beq
	\label{eq:hess_mat} 
	\begin{bmatrix} 0 & G(x,y) \\ G(x,y)^* & 0 \end{bmatrix},
\eeq
$\nabla^2 g(\hat x, \hat y)$ can be computed by applying \cref{thm:eig2ndderiv} to \eqref{eq:hess_mat}.
Computationally, the necessary first and second partial derivatives of \eqref{eq:hess_mat} can be obtained via the first partials given in \eqref{eq:first_partials_cont} and the following second 
partial derivatives:
\beq
	\label{eq:second_partials_cont}
	\frac{\partial^2 G(x,y)}{\partial x^2} = \frac{-2(A - \imagunit y I) }{x^3},
	\quad 
	\frac{\partial^2 G(x,y)}{\partial y^2} = 0,
	\quad \text{and} \quad 
	\frac{\partial^2 G(x,y)}{\partial x \partial y} = \frac{-\imagunit I}{x^2}.
\eeq
Although the full eigendecomposition of \eqref{eq:hess_mat} is needed,
it can actually be constructed more or less for free given the full SVD of $G(\hat x, \hat y)$; 
see \cite[Section~2.2]{BenM19} for details.
If $(\hat x, \hat y)$ is additionally a distinct (up to conjugacy) \emph{global} minimizer of $g(x,y)$, 
then the gradient and Hessian of $\Kinv$ with respect to $A$ is equivalent to the gradient and Hessian of $g(\hat x, \hat y)$
with respect to $A$. 
The gradient and Hessian of $\Kcon$ is then simply obtained by applying the
chain rule for the inverse.

The cost of obtaining $g(\hat x,\hat y)$, its gradient, and its Hessian is
$\bigO(\n^3)$, as they can all be computed given the full SVD of $G(\hat x, \hat y)$.
Although \eqref{eq:kinv_cont} is technically a constrained optimization problem, 
$g(x,y) \to \infty$ as $x$ approaches zero from the right, assuming that $\imagunit y$ 
is not an eigenvalue of $A$.
Thus, just returning $\infty$ as the value of $g(x,y)$ whenever $x \leq 0$
suffices for using unconstrained optimization solvers to find feasible local/global minimizers of \eqref{eq:kinv_cont}.
Provided $g(x,y)$ is sufficiently smooth about its stationary points,
one can expect local quadratic convergence when using a Newton-based optimization method and superlinear convergence with a quasi-Newton method 
(forgoing the use of the Hessian).
Note that scalable methods for computing smallest singular values, e.g., PROPACK \cite{propack},
can also be used to compute $g(\hat x,\hat y)$ and its associated pair of left and right singular vectors in order to obtain $\nabla g(\hat x, \hat y)$.
Thus, by combining such a sparse solver with a quasi-Newton method, one can 
efficiently obtain locally optimal approximations to Kreiss constants of large-scale matrices.

\begin{remark}
One could also consider using optimization to find maximizers of \eqref{eq:k1d_cont},
which has the benefit of working with only one optimization variable instead of two.
However, computing $\aleps(A)$ 
is substantially more expensive than the minimum singular value of a matrix; 
the quadratically-convergent criss-cross algorithm of \cite{BurLO03}
to compute $\aleps(A)$, as well as the faster method of \cite{BenM19}, 
require computing all eigenvalues of $2\n \times 2\n$ matrices, often several times.
Moreover, we have just shown how the explicit Hessian of $g(x,y)$ can easily be computed 
in order to obtain faster convergence of the iterates produced by optimization methods.
Finally, for large-scale $A$ matrices, sparse methods for $\smin(A)$
are generally much faster and more reliable than 
those for approximating $\aleps(A)$ \cite{GugO11,KreV14}.
\end{remark}

That $g(x,y)$ is not so dissimilar to $f(x,y)$ for $\dtu$ indicates that
it might be possible to adapt Gu's verification procedure to develop globality certificates for $g(x,y)$.
Combined with the optimization techniques discussed here, this would enable
a globally convergent optimization-with-restarts method for Kreiss constants
that terminates within a finite number of restarts.

\subsection{The discrete-time case}
\label{sec:opt_disc}
Again identifying $\C$ with $\R^2$, but now using polar coordinates, 
consider the inverse of the discrete-time Kreiss constant \eqref{eq:k2d_disc}, i.e.,
\beq	
	\label{eq:kinv_disc}
	\Kinv = \inf_{r > 1, \theta \in [0,2\pi)} 
	\smin \left(\frac{r\eit I - A}{r - 1} \right)
 	=  \inf_{r > 1, \theta \in [0,2\pi)} h(r,\theta),
\eeq
where
\beq
	\label{eq:h_disc}
	h(r,\theta) = \smin(H(r,\theta)) \quad \text{and} \quad H(r,\theta) = \frac{r\eit I - A}{ r - 1}.
\eeq
Naturally $h(r,\theta)$ has the same key properties as $g(x,y)$, i.e., it too is semialgebraic and locally Lipschitz.
Thus, $h(r,\theta)$ has a finite number of locally minimal function values,
and we can consider using optimization to find minimizers of $h(r,\theta)$.
We will need the analogous gradients and Hessian of $h(r,\theta)$;
for brevity, we just provide the first and second partial derivatives of $H(r,\theta)$ for $r \ne 1$ here,
which are respectively
\beq
	\label{eq:first_partials_disc}
	\frac{\partial H(r,\theta)}{\partial r}  
	= \frac{(r-1)\eit I - (r\eit I - A)}{(r-1)^2} 
	= \frac{A - \eit I }{(r - 1)^2}
	\quad \text{and} \quad
	\frac{\partial H(r,\theta)}{\partial \theta} = \frac{\imagunit r\eit I}{r - 1},
\eeq
and
\beq
	\label{eq:second_partials_disc}
	\frac{\partial^2 H(r,\theta)}{\partial r^2} = \frac{-2(A - \eit I) }{(r-1)^3},
	\ 
	\frac{\partial^2 H(r,\theta)}{\partial \theta^2} = \frac{-r\eit I}{r-1},
	 \  \text{and} \ 
	\frac{\partial^2 H(r,\theta)}{\partial r \partial \theta} = \frac{-\imagunit \eit I}{(r-1)^2}.
\eeq
The costs to compute $h(r,\theta)$ along with its gradient and Hessian also remain as are described in \S\ref{sec:opt_cont} for $g(x,y)$, and the variational results above similarly allow the gradient and Hessian 
of $\Kinv$ or $\Kcon$ to be computed in the discrete-time case.
To ensure that optimization returns a feasible minimizer to \eqref{eq:kinv_disc}, 
it suffices to return $\infty$ for the value of $h(r,\theta)$ whenever $r \leq 1$;
this is because for $\eit$ not an eigenvalue of $A$, $\lim_{r \to 1^+} h(r,\theta) \to \infty$.
Thus, locally optimal approximations to discrete-time Kreiss constants can be computed efficiently,
for small- or large-scale matrices.
To develop a globally convergent algorithm, we will need to develop a discrete-time globality certificate.

\section{Continuous-time Kreiss constant analogues of Gu's theorem and their consequences}
\label{sec:analogues}
Before we present our globally convergent iterations for computing Kreiss constants,
we first develop analogues of \cref{thm:dtu,cor:dtu}.
For the time being, we consider continuous-time $\Kcon$ and begin by considering 
vertically oriented pairs of points on the $\gamma$-level set of $g(x,y)$.
For $\dtu$, Gu considered pairs of level-set points oriented horizontally, but this choice was rather arbitrary. 
However, as we will soon see, for Kreiss constants the choice of orientation does have important consequences,
both theoretically and for our new algorithms.

\begin{theorem}
\label{thm:kinv_cv}
For $A \in \C^{\n \times \n}$ with $\alpha(A) < 0$,
let $\gamma \in [0,1)$,
$\eta \geq 0$, and $(x_\star,y_\star)$ be a global minimizer of \eqref{eq:kinv_cont}.
If $\Kinv \leq \gamma$ and $\eta \in [0,2x_\star(\gamma - \Kinv)]$, then 
there exists a pair $x,y \in \R$ with $x > 0$ such that
\beq
	\label{eq:kinv_cv_ub}
	g(x,y) = g(x,y+\eta) = \gamma.
\eeq
\end{theorem}

\begin{corollary}
\label{cor:kinv_cv}
For $A \in \C^{\n \times \n}$ with $\alpha(A) < 0$, let $\gamma \in [0,1)$, $\eta \geq 0$,
and $(x_\star,y_\star)$ be a global minimizer of \eqref{eq:kinv_cont}.
If there do not exist any pairs $x,y \in \R$ with $x > 0$ such that \eqref{eq:kinv_cv_ub} holds,
then 
\beq
	\label{eq:kinv_cv_lb}
	\Kinv > \gamma - \tfrac{\eta}{2x_\star}.
\eeq
\end{corollary}

To prove \cref{thm:kinv_cv}, we will use the following topology definition and result.
\begin{definition}
Let $\mathcal{A} \subset \C$ be a bounded open (path) connected set and let $\mathcal{A}^\mathrm{C}$
be its complement.
Furthermore, let $\mathcal{B} = \{ z \in \C: z \text{ is in a bounded component of } \mathcal{A}^\mathrm{C} \}$.
Then $\mathcal{A} \cup \mathcal{B}$ is the \emph{simply connected hull} of $\mathcal{A}$, which we denote $\mathcal{A}^\mathrm{H}$.
\end{definition}

\begin{lemma}
\label{lem:topoboundaries}
Let $\mathcal{A},\mathcal{B} \subset \C$ both be bounded open (path) connected sets.
If there exist points $b_\mathrm{in}, b_\mathrm{out} \in \bd \mathcal{B}^\mathrm{H}$ such that $b_\mathrm{in} \in \mathcal{A}$
and $b_\mathrm{out} \in \mathcal{A}^\mathrm{C}$, then $\bd \mathcal{A} \cap \bd \mathcal{B} \neq \varnothing$.
\end{lemma}
\begin{proof}
Since $\bd \mathcal{B}^\mathrm{H} \subset \bd \mathcal{B}$, 
we can assume that $b_\mathrm{out}$ is in the interior of $\mathcal{A}^\mathrm{C}$, as otherwise the proof is done.
As $\mathcal{B}^\mathrm{H}$ is a bounded open simply connected set,
$\bd \mathcal{B}^\mathrm{H}$ must be connected; see, e.g., \cite[p.~345]{Bur79}.
By way of contradiction, suppose that $\bd \mathcal{A} \cap \bd \mathcal{B} = \varnothing$. 
Then $\bd \mathcal{B}^\mathrm{H} \subset \bd \mathcal{B} \subset \mathcal{A} \cup \mathrm{int} (\mathcal{A}^\mathrm{C})$.
However, since $b_\mathrm{in} \in \mathcal{A}$, $b_\mathrm{out} \in \mathrm{int} (\mathcal{A}^\mathrm{C})$,
and $\mathcal{A}$ and $\mathrm{int} (\mathcal{A}^\mathrm{C})$ are both open and nonempty disjoint sets,
it follows by definition of connected that $\bd \mathcal{B}^\mathrm{H}$ is in fact disconnected, a contradiction.
Thus, $\bd \mathcal{A} \cap \bd \mathcal{B} \neq \varnothing$ holds.
\end{proof}

\begin{proof}[Proof of \cref{thm:kinv_cv}]
If $\gamma=\Kinv$, the proof is trivially satisfied with $\eta = 0$, so assume that $\gamma \in (\Kinv,1)$.
Since $\alpha(A) < 0$, it follows that $\lim_{x\to0^+} g(x,y) = \infty$ for any $y \in \R$,
and so $g(x_\star,y_\star) = \Kinv$ with $x_\star > 0$.
Now consider the strict lower level set $\mathcal{L}_\gamma \coloneqq \{ (x,y) : g(x,y) < \gamma, x > 0\}$,
which is clearly open and also bounded; see \cite[Theorem~2.3]{Mit19a}.
Let $\mathcal{L}$ be the (open) connected component of $\mathcal{L}_\gamma$ such that
$(x_\star,y_\star) \in \mathcal{L}$ and let $\mathcal{G} \coloneqq \mathcal{L}^\mathrm{H}$, i.e., the 
simply connected hull of $\mathcal{L}$.

Now following the proof of \cite[Theorem~3.1]{Gu00} a bit more closely,
by continuity of $g(x,y)$ and the fact that $\mathcal{G}$ is bounded, 
there must exist points $b_1,b_2 \in \bd \mathcal{G}$
\[
	b_1 = (x_\star, y_\star - \eta_1) 
	\qquad \text{and} \qquad
 	b_2 = (x_\star, y_\star + \eta_2)
\]
such that 
\beq
	\label{eq:shifted}
	g(x_\star,y_\star - \eta_1) = g(x_\star,y_\star + \eta_2) = \gamma,
\eeq
where $\eta_1,\eta_2 > 0$.
Furthermore, we can assume that $\eta_1$ and $\eta_2$ are the smallest positive values 
such that \eqref{eq:shifted} holds with $b_1,b_2 \in \bd \mathcal{G}$.
Noting that 
\beq
	\label{eq:f_numerator}
	g(x,y) =\frac{\smin\left((x+\imagunit y) I - A\right)}{x} = \frac{f(x,y)}{x},
\eeq
whose numerator has a GLC of 1, it follows for any $y_1,y_2 \in \R$ that
\[
	| g(x_\star,y_1) - g(x_\star,y_2) | \leq \tfrac{1}{x_\star} | y_1 - y_2 |,
\]
i.e., $g(x_\star,y)$ with respect to $y$ has a GLC of $\tfrac{1}{x_\star}$.  
By applying this GLC to \eqref{eq:shifted}, it follows that 
\beq
	\label{eq:eta12}
	\eta_1 \geq x_\star (\gamma - \Kinv) 
	\quad \text{and} \quad \eta_2 \geq x_\star (\gamma - \Kinv),
\eeq
and so $\eta_1 + \eta_2 \geq 2x_\star(\gamma - \Kinv)$.

Now suppose that $\eta \in (0,2x_\star(\gamma - \Kinv)]$ so $\eta \leq \eta_1 + \eta_2$.
Obviously \eqref{eq:kinv_cv_ub} is satisfied if $\eta = \eta_1 + \eta_2$, 
so assume that $\eta < \eta_1 + \eta_2$.
Let
\[
	\mathcal{G}_\eta \coloneqq \{ (x,y - \eta) : (x,y) \in \mathcal{G}\},
\]
i.e., $\mathcal{G}$ shifted downward by the amount $\eta$,
and consider the line segment joining $b_1$ and $b_2$ and
the point $b_\mathrm{in} = (x_\star,y_\star + \eta_2 - \eta)$.
As $b_\mathrm{in}$ must be on this line segment, but not at its endpoints, 
$b_\mathrm{in} \in \mathcal{G}$ and 
$b_\mathrm{in} \in \bd \mathcal{G}_\eta$, since $b_2 \in \bd \mathcal{G}$.
Let $b_\mathrm{out} \in \bd \mathcal{G}_\eta$ be a lowermost point of $\bd \mathcal{G}_\eta$.
Then, as $\eta > 0$, $b_\mathrm{out} \in \mathcal{G}^\mathrm{C}$, the complement of $G$.
Since $\mathcal{G}$ and $\mathcal{G}_\eta$ are both bounded open connected sets in the plane
and $\mathcal{G}_\eta = \mathcal{G}_\eta^\mathrm{H}$, \cref{lem:topoboundaries} applies,
and so $\bd \mathcal{G} \cap \bd \mathcal{G}_\eta \neq \varnothing$.
Letting $(\tilde x, \tilde y)$ be any such point in $\bd \mathcal{G} \cap \bd \mathcal{G}_\eta$,
it follows that $(\tilde x, \tilde y + \eta) \in \bd \mathcal{G}$,
hence $(\tilde x,\tilde y)$ satisfies \eqref{eq:kinv_cv_ub}.
\end{proof}

We now consider horizontally oriented pairs of points on the $\gamma$-level set of $g(x,y)$, 
similar to \cref{thm:dtu} for $f(x,y)$.
As the horizontal orientation is actually more complicated for $\Kcon$, 
we first establish the following intermediate general result.

\begin{lemma}
\label{lem:foverx}
Let $f : (0,\infty) \to [0,\infty)$ be continuous with a GLC of $c \geq 0$
and consider the function $q(x) \coloneqq \tfrac{f(x)}{x}$ on the same domain.  
For $a,b > 0$ with $b - a = \eta > 0$, if $q(a) = q(b) = \gamma$ 
and $x_\star = \argmin_{x \in [a,b]} q(x)$ with $\gamma_\star = q(x_\star)$, then 
\[
	\gamma_\star \geq \gamma - \tfrac{\eta(c+\gamma)}{2x_\star}.
\]
\end{lemma}
\begin{proof}
We assume that $\gamma_\star < \gamma$ as otherwise the inequality clearly holds.
Since $f(x)$ has a GLC of $c$, it follows that 
\beqs
	\label{eq:glc_ub}
	f(a) - f(x_\star) + f(b) - f(x_\star) \leq c(x_\star - a) + c(b - x_\star) = c(b-a) = c\eta.
\eeqs
Meanwhile
\beqs
	\label{eq:glc_lb}
	f(a) - f(x_\star) + f(b) - f(x_\star) = a\gamma + b\gamma - 2x_\star \gamma_\star \geq 
	(a + 2x_\star -b) \gamma - 2x_\star \gamma_\star = 2x_\star (\gamma - \gamma_\star) - \eta \gamma.
\eeqs
Combining the two yields $2x_\star (\gamma - \gamma_\star) - \eta \gamma \leq c\eta$,
thus completing the proof.
\end{proof}

\begin{theorem}
\label{thm:kinv_ch}
For $A \in \C^{\n \times \n}$ with $\alpha(A) < 0$, let $\gamma \in [0,1)$, $\eta \geq 0$,
and $(x_\star,y_\star)$ be a global minimizer of \eqref{eq:kinv_cont}.
If $\Kinv \leq \gamma$ and $\eta \in \big[0,\tfrac{2x_\star}{1 +\gamma}(\gamma - \Kinv)\big]$, then 
there exists a pair $x,y \in \R$ with $x > 0$ such that
\beq
	\label{eq:kinv_ch_ub}
	g(x,y) = g(x + \eta,y) = \gamma.
\eeq
\end{theorem}

\begin{corollary}
\label{cor:kinv_ch}
For $A \in \C^{\n \times \n}$ with $\alpha(A) < 0$, let $\gamma \in [0,1)$, $\eta \geq 0$, 
and $(x_\star,y_\star)$ be a global minimizer of \eqref{eq:kinv_cont}.
If there do not exist any pairs $x,y \in \R$ with $x > 0$ such that \eqref{eq:kinv_ch_ub} holds,
then 
\beq
	\label{eq:kinv_ch_lb}
	\Kinv > \gamma - \tfrac{\eta(1 + \gamma)}{2x_\star}.
\eeq
\end{corollary}

\begin{proof}[Proof of \cref{thm:kinv_ch}]
The beginning of the proof is the same as the first paragraph
of the proof of \cref{thm:kinv_cv}.
Again consider the simply connected set $\mathcal{G}$ defined there
that contains $(x_\star,y_\star)$, a global minimizer.
By continuity of $g(x,y)$ and the fact that $\mathcal{G}$ is bounded, there must exist points
$b_1,b_2 \in \bd \mathcal{G}$ 
\[
	b_1 = (x_\star - \eta_1, y_\star ) 
	\qquad \text{and} \qquad
 	b_2 = (x_\star + \eta_2, y_\star)
\]
such that
\beq
	\label{eq:shifted_horz}
	g(x_\star - \eta_1, y_\star) = g(x_\star + \eta_2, y_\star) = \gamma,
\eeq
where $\eta_1,\eta_2 > 0$.
We again assume that $\eta_1$ and $\eta_2$ are the smallest positive values 
such that \eqref{eq:shifted_horz} holds with $b_1,b_2 \in \bd \mathcal{G}$.
Applying \cref{lem:foverx} to $q(x) \coloneqq g(x,y_\star) = \tfrac{f(x,y_\star)}{x}$ with $a = x_\star - \eta_1$ and $b = x_\star +\eta_2$,
we have that 
\[
	\eta_1 + \eta_2 \geq \tfrac{2x_\star}{1 + \gamma}(\gamma - \Kinv)
\]
as the numerator of $g(x,y)$ as rewritten in \eqref{eq:f_numerator} has a GLC of 1.

Now suppose that $\eta \in \big(0, \tfrac{2x_\star}{1 + \gamma}(\gamma - \Kinv)\big]$ so $\eta \leq \eta_1 + \eta_2$.
If $\eta = \eta_1 + \eta_2$, \eqref{eq:kinv_ch_ub} is clearly satisfied, 
so instead assume that $\eta < \eta_1 + \eta_2$.
Considering the set
$
	\mathcal{\widehat G} \coloneqq \{ (x - \eta,y) : (x,y) \in \mathcal{G}\},
$
i.e., $\mathcal{G}$ shifted to the left by the amount $\eta$,
the rest of the proof follows analogously to the end of the proof of \cref{thm:kinv_cv}.
\end{proof}

Although we have derived Kreiss constant analogues of \cref{thm:dtu,cor:dtu},
\cref{cor:kinv_cv,cor:kinv_ch} in fact assert that the $\dtu$ algorithms of \cite{Gu00,BurLO04,GuMOetal06}
will \emph{not} directly extend to Kreiss constants.  The crux of the problem 
is that these $\dtu$ methods all rely on the fact that \eqref{eq:dtu_lb} holds 
when there are no points satisfying \eqref{eq:dtu_ub},
which recall, provided a way of computing $\dtu$ via Gu's verification procedure 
to verify an upper or lower bound for $\dtu$.
However, in the Kreiss constant setting, our lower bounds given in \eqref{eq:kinv_cv_lb} and \eqref{eq:kinv_ch_lb}
are not as concrete, as they depend on $x_\star$, which is unknown.
For a given $\gamma$ and $\eta$, we do not even know if 
the lower bounds provided \eqref{eq:kinv_cv_lb} and \eqref{eq:kinv_ch_lb} would be meaningful,
as they might not even be positively valued.
Thus, to develop algorithms for $\Kcon$, crucial departures must be made.
We will do this via two different strategies.

\section{A continuous-time $\Kcon$ algorithm based on fixed-distance pairs}
\label{sec:cont}
Let $(\hat x, \hat y)$ be a local (but not global) minimizer of $g(x,y)$ and set
$\gamma \coloneqq g(\hat x, \hat y)$.
As $(\hat x, \hat y)$ is a local minimizer, we do not need to verify that $\gamma \geq \Kinv$,
 as this is obviously true.  To obtain an optimization-with-restarts algorithm, we do not necessarily 
 need to verify a lower bound either.  Instead, we can just aim to detect 
other (non-stationary) points on the $\gamma$-level set of $g(x,y)$.
Using such level-set points to restart optimization, a better (lower) minimizer
of $g(x,y)$ is guaranteed to be found.
Assuming $\gamma \in [\Kinv,1)$, \cref{thm:kinv_cv,thm:kinv_ch} assert that if 
 $\eta > 0$ is chosen sufficiently small, there must exist points satisfying either \eqref{eq:kinv_cv_ub}
or \eqref{eq:kinv_ch_ub}.
Of course, we do not know \emph{a priori} how small to choose $\eta$, so we propose to 
use backtracking, i.e., we can start with $\eta$ initially set to some large value and 
simply decrease it in a loop until level-set points for restarting optimization are found.
As long as $(\hat x, \hat y)$ is not a global minimizer, this backtracking procedure must succeed 
in finding level-set points for restarting optimization. 
Meanwhile, when $(\hat x, \hat y)$ is a global minimizer, and so $\Kcon$ has been computed,
the backtracking procedure can simply be terminated once $\eta$ falls below a tolerance.
 A high-level pseudocode using this backtracking-based globality certificate is given in \cref{alg:owr_bt}.
 To complete this algorithm, we now must develop a corresponding 2D level-set test for continuous-time Kreiss constants.
We will do this by looking for points 
satisfying \eqref{eq:kinv_cv_ub} or \eqref{eq:kinv_ch_ub},
i.e., level-set points that are a \emph{fixed distance} $\eta$ apart,
and develop a procedure inspired by Gu's 2D level-set test for~$\dtu$.

\begin{algfloat}[t]
\begin{algorithm}[H]
\floatname{algorithm}{Algorithm}
\caption{Optimization-with-restarts using backtracking}
\label{alg:owr_bt}
\begin{algorithmic}[1]
	\REQUIRE{  
		A nonnormal matrix $A \in \C^{n \times n}$ with $\alpha(A) < 0$,
		$x_0 > 0$ and $y_0 \in \R$ such that $g(x_0,y_0) < 1$,
		and a tolerance $\eta_\mathrm{tol} > 0$.
		}
	\ENSURE{ 
		$\gamma^{-1} \approx \Kcon$ (continuous-time).
		\\ \quad
	}
	
	\WHILE { true } 
		\STATE $(\hat x, \hat y) \gets $ computed local/global minimizer of \eqref{eq:kinv_cont} initialized from $(x_0,y_0)$
		\STATE $\gamma \gets g(\hat x, \hat y)$
		\STATE $\eta \gets $ some positive value $\gg \eta_\mathrm{tol}$
		\WHILE { true } 
			\STATE Perform the 2D level-set test of \S\ref{sec:2d_test_fixed} with current $\gamma$ and $\eta$
			\IF { test finds \emph{any} level-set points }
			 	\STATE $(x_0,y_0) \gets$ one of these points
				\BREAK $\, \qquad$ \COMMENT{Goto line 2 to restart optimization.}
			\ELSIF { $\eta \leq \eta_\mathrm{tol}$ }
				\RETURN $\quad$ \COMMENT{Found a global minimizer to tolerance.}
			\ELSE
				\STATE $\eta \gets c \eta$ for some constant $c \in (0,1)$
			\ENDIF
		\ENDWHILE
	\ENDWHILE
\end{algorithmic}
\end{algorithm}
\algnote{
For simplicity of the pseudocodes, we assume here and in \cref{alg:owr} that (a) optimization 
always converges to local or global minimizers exactly, i.e., not approximately or to other stationary points,
and (b) points found by the certificate test (if any) are never exactly stationary.
}
\end{algfloat}

\subsection{A 1D vertical level-set test}
\label{sec:1d_test}
Before we develop our 2D level-set test for $g(x,y)$, we will need the following theorem 
which will allow us to obtain all the points on the $\gamma$-level set of $g(x,y)$
along a chosen vertical line.
\begin{theorem}
\label{thm:1d_vert}
Given $\gamma, x,y \in \R$, with $\gamma \geq 0$ and $x \neq 0$,
$\gamma$ is a singular value of $G(x,y)$ defined in \eqref{eq:g_cont} if and only if $\imagunit y$ is an eigenvalue of
the Hamiltonian matrix 
\beq
\label{eq:1d_vert_mat}
\begin{bmatrix} 
	A - xI   		& \gamma xI \\
	-\gamma xI 	& xI - A^*
\end{bmatrix}.
\eeq
\end{theorem}
\begin{proof}
It is clear that the matrix is Hamiltonian.
Suppose $\gamma$ is a singular value $G(x,y)$ with left and right singular vectors $u$ and $v$,
and so
\[
\thinmuskip=2.5mu
\medmuskip=3.5mu
\thickmuskip=4.5mu
	\gamma 
	\begin{bmatrix} u \\ v \end{bmatrix}
	=
	\begin{bmatrix} 
		G(x,y) 	& 0 	\\
		0 		& G(x,y)^* 
	\end{bmatrix}
	\begin{bmatrix} v \\ u \end{bmatrix}
	\ \Leftrightarrow \
	\gamma 
	x
	\begin{bmatrix} u \\ v \end{bmatrix}
	=
	\begin{bmatrix} 
		(x + \imagunit y)I - A 	& 0\\
		0				& (x - \imagunit y)I - A^* 	
	\end{bmatrix}
	\begin{bmatrix} v \\ u \end{bmatrix}.	
\]
Rearranging terms and multiplying the bottom block row by $-1$,
this is equivalent to
\[
\thinmuskip=1.9mu
\medmuskip=2.9mu
\thickmuskip=3.9mu
	\gamma 
	x
	\begin{bmatrix} u \\ {-}v \end{bmatrix}
	+
	\begin{bmatrix} 
		A - xI 	& 0\\
		0		& xI - A^* 	
	\end{bmatrix}
	\begin{bmatrix} v \\ u \end{bmatrix}
	= 
	\imagunit y
	\begin{bmatrix} v \\ u \end{bmatrix}
	\ \Leftrightarrow \
	\begin{bmatrix} 
		A - xI 		& \gamma xI \\
		-\gamma xI	& xI - A^* 	
	\end{bmatrix}
	\begin{bmatrix} v \\ u \end{bmatrix}
	= 
	\imagunit y
	\begin{bmatrix} v \\ u \end{bmatrix}.
\]
\end{proof}

In fact, for a given $\gamma \geq \Kinv$ and $x \coloneqq \hat x \neq 0$, 
computing the imaginary eigenvalues of \eqref{eq:1d_vert_mat} may provide more than 
the $\gamma$-level set points of $g(x,y)$ along the vertical line $x = \hat x$.
This is because if $\imagunit \hat y$ is an eigenvalue of \eqref{eq:1d_vert_mat},
\cref{thm:1d_vert} asserts that $\gamma$ is a singular value of $F(\hat x,\hat y)$,
but not necessarily the minimum one; if this happens,
$(\hat x, \hat y)$ would be on a level set of $g(x,y)$ lower than the $\gamma$-level set 
and thus be an even better point for restarting optimization.
Finally, note that structure-preserving eigensolvers \cite{BenMX98a,BenMX98b} 
can be used for more reliable detection of imaginary eigenvalues of Hamiltonian matrices
like \eqref{eq:1d_vert_mat}.

\subsection{A 2D level-set test for fixed-distance pairs}
\label{sec:2d_test_fixed}
We now derive a new 2D level-set test for $g(x,y)$.  
Per \cref{thm:kinv_cv,thm:kinv_ch}, 
the choice of orientation for pairs of points on the $\gamma$-level set of $g(x,y)$ has consequences,
hence we will develop our new continuous-time Kreiss constant procedure for arbitrary orientation.
Specifically, given $\eta > 0$ and angle $\theta \in (-\tfrac{\pi}{2}, \tfrac{\pi}{2}]$, 
we will look for points a fixed distance $\eta$ apart of the form 
$(\hat x, \hat y)$ and $(\hat x + \eta \costh, \hat y + \eta \sinth)$
such that $g(x,y) = \gamma$ holds at both of them and $\hat x > 0$.

Suppose that $\gamma$ is a singular value of both $G(x,y)$ and $G(x+\eta\costh,y+\eta\sinth)$, with respective left and right singular vectors pairs $u$,$v$ and $\hat u$,$\hat v$.
By \cref{thm:1d_vert} for $G(x,y)$ and following a similar argument as in its proof for 
$G(x+\eta\costh,y+\eta\sinth)$, the following two Hamiltonian matrices must share an imaginary eigenvalue
$\imagunit y$, i.e., 
\bseq
\label{eq:eig_pair_ctheta}
\begin{align}
	\label{eq:eig_pair_ctheta_A1}
	\begin{bmatrix}
		A - xI   		& \gamma xI \\
		-\gamma xI 	& xI - A^*
	\end{bmatrix}
	\begin{bmatrix} v \\ u \end{bmatrix}
	&= \imagunit y \begin{bmatrix} v \\ u \end{bmatrix} \\
	\label{eq:eig_pair_ctheta_A2}
	\begin{bmatrix}
		A - (x+\eta\eit)I   		& \gamma (x+\eta\costh) I \\
		-\gamma (x+\eta\costh) I 	& (x+\eta\emit)I - A^*
	\end{bmatrix}
	\begin{bmatrix} \hat v \\ \hat u \end{bmatrix}
	&= \imagunit y \begin{bmatrix} \hat v \\ \hat u \end{bmatrix}.
\end{align}
\eseq
Let $A_1$ and $A_2$ respectively denote the square matrices in \eqref{eq:eig_pair_ctheta_A1}
and \eqref{eq:eig_pair_ctheta_A2} and 
let $W = \begin{bsmallmatrix} v \\ u\end{bsmallmatrix} \begin{bsmallmatrix} \hat v^* & \hat u ^* \end{bsmallmatrix} \neq 0$ so that we have $A_1W = \imagunit y W$ and $A_2W^* = \imagunit y W^*$.  
To eliminate~$y$, we take the conjugate transpose of the second 
and then add the two together to obtain the Sylvester equation:
\beq
	\label{eq:sylv_ctheta}
	\begin{bmatrix}
		A - xI   		& \gamma xI \\
		-\gamma xI 	& xI - A^*
	\end{bmatrix}
	W + W
	\begin{bmatrix}
		A^* - (x+\eta\emit)I  		& -\gamma (x+\eta\costh) I \\
		\gamma (x+\eta\costh) I 	& (x+\eta\eit)I - A
	\end{bmatrix}
	= 0.
\eeq
Thus, if \eqref{eq:sylv_ctheta} also has a nonzero solution $W \in \C^{2\n\times2\n}$,
$A_1$ and $-A_2^*$ must share an eigenvalue.
As $A_1$ and $A_2$ are Hamiltonian matrices, their spectra
have imaginary-axis symmetry, and so $A_1$ and $A_2$
must also have an eigenvalue in common.
Now separating out all terms involving $x$, we get
\begin{multline}
	\label{eq:sylv_ctheta_x}
	\left(
	\begin{bmatrix} 
		A 	& 0 \\ 
		0 	& -A^*
	\end{bmatrix}
	W + W
	\begin{bmatrix} 
		A^* - \eta \emit I  		& -\gamma \eta \costh I \\
		\gamma \eta \costh I 		& \eta \eit I - A
	\end{bmatrix} 
	\right) \\
	- x 
	\left( 
	\begin{bmatrix} 
		I 		& -\gamma I \\ 
		\gamma I  & -I 
	\end{bmatrix} 
	W + W
	\begin{bmatrix} 
		I 		& \gamma I \\
		-\gamma I  & -I 
	\end{bmatrix}
	\right) = 0.
\end{multline}
Rewriting both Sylvester forms using the vectorize operator, 
and letting $w = \vecop(W)$, results in the generalized eigenvalue problem
\begin{align}
	\label{eq:large_eig_ctheta}
	\mathcal{A}_1 w &= x \mathcal{A}_2 w, 
	\quad \text{where} \\
	\nonumber
	\mathcal{A}_1 &=  
	I_{2n} \otimes 
	\begin{bmatrix} 
		A 	& 0 \\ 
		0 	& -A^*
	\end{bmatrix}
	+
	\begin{bmatrix} 
		\, \overline{A} - \eta \emit I  	& \gamma \eta \costh I \\
		-\gamma \eta \costh I 		& \eta \eit I - A\tp
	\end{bmatrix}
	\otimes I_{2n}, \\
	\nonumber
 	\mathcal{A}_2 &= 
	I_{2n} \otimes 
	\begin{bmatrix} 
		I 		& -\gamma I \\ 
		\gamma I  & -I 
	\end{bmatrix} 
	+
	\begin{bmatrix} 
		I 		& -\gamma I \\
		\gamma I  & -I 
	\end{bmatrix}
	\otimes I_{2n}.
\end{align}

Our 2D level-set test for fixed-distance pairs works as follows.
Given $\gamma \in [\Kinv,1)$, $\eta \geq 0$, and $\theta \in (-\tfrac{\pi}{2},\tfrac{\pi}{2}]$,
we first compute all the eigenvalues of \eqref{eq:large_eig_ctheta}.
If there are no (finite) positive real eigenvalues of \eqref{eq:large_eig_ctheta},
then the test is finished and returns no level-set points.
Otherwise, 
additional calculations are done to ascertain whether or not any level-set points have been detected.
For each eigenvalue $\hat x > 0$ of \eqref{eq:large_eig_ctheta}, we have that 
the vertical line specified by $\hat x$, and, if $|\theta| < \tfrac{\pi}{2}$,
the vertical line $\hat x + \eta \cos \theta$, \emph{may} contain points on the 
$\gamma$-level set of $g(x,y)$.
To determine this, for each of these vertical lines, say, $x \coloneqq \hat x$, 
we then apply \cref{thm:1d_vert} and compute all the eigenvalues 
of the corresponding Hamiltonian matrix \eqref{eq:1d_vert_mat}.
If this matrix has no imaginary eigenvalues, then no level-set points have been detected
on this vertical line.  Otherwise, all the points $(\hat x, \hat y)$ such that $\imagunit \hat y$
is an imaginary eigenvalue of this matrix are added to the list of detected level-set points of $g(x,y)$ to return.
Optionally, for each of these points, one could additionally check whether or not $\gamma$ is the minimum singular value of $G(\hat x,\hat y)$, but this is not strictly necessary; 
as discussed previously, $(\hat x, \hat y)$ must be on the $\gamma$-level set of $g(x,y)$ or a lower one,
either of which suffice for restarting optimization to obtain a better (lower) minimizer of $g(x,y)$.
If no level-set points are detected, for any of the vertical lines, then the test returns no points.
Otherwise, all the detected level-set points are returned. 

Our new 2D level-set test differs from the procedure in \cite{Gu00} (and \cite{GuMOetal06}) in a significant way;
in \cref{rem:improved_test}, we explain how our modifications here 
greatly improve the reliability of these 2D level-set tests.

In terms of cost, in extreme situations there may be up to $\bigO(\n^2)$ potential vertical lines detected,
which means that $\bigO(\n^2)$ Hamiltonian eigenvalue problems of dimension $2\n \times 2\n$
must be solved.  However, when using standard dense eigensolvers, the overall 
work complexity of our procedure is actually $\bigO(\n^6)$, 
as we must first compute the eigenvalues of \eqref{eq:large_eig_ctheta}, 
which is a matrix pencil with square matrices of dimension $4\n^2$.
In terms of constant factors, if $A$ is real, $\mathcal{A}_1$ is real if and only if
$\theta = 0$.

\subsection{Properties of the eigensystem $\mathcal{A}_1w = x \mathcal{A}_2w$ and its solution}
\label{sec:2d_test_fixed_props}
Unlike the $4\n^2 \times 4\n^2$ generalized eigenvalue problem
\eqref{eq:large_eig_ctheta}, 
Gu derived a smaller $2\n^2 \times 2\n^2$ generalized eigenvalue problem for $\dtu$ \cite[Eq.~3.13]{Gu00}.
In \cite[Section~3.1]{GuMOetal06}, this was then simplified further to a computationally easier $2\n^2 \times 2\n^2$ standard eigenvalue problem \cite[Eq.~3.7]{GuMOetal06}.
Via the following result, we show that the rank of $\mathcal{A}_2$ is $2n^2$, i.e., half its dimension.

\begin{lemma}
\label{lem:UkVk}
Let $C \coloneqq \begin{bsmallmatrix} aI & -bI \\ bI & -aI\end{bsmallmatrix} \in \C^{2n \times 2n}$ with $a,b \in \C$ and $b \neq 0$.
For $k \in \N$,
the Kronecker sum 
$I_{2k} \otimes C + C \otimes I_{2k}  = \mathcal{U}_k\mathcal{V}_k \in \C^{4kn \times 4kn}$
and $\mathcal{V}_k \mathcal{U}_k = 2I_k \otimes C$, where
\bseq
	\label{eq:UkVk}
	\begin{align}
	\mathcal{U}_k &\coloneqq
	\left[
	\renewcommand*{\arraystretch}{1.3}
	\begin{array}{c}	
	I_{k} \otimes 
	\begin{bsmallmatrix}
	2aI 	& -bI \\
	bI 	& 0 \\
	\end{bsmallmatrix} \\
	\hdashline
	b I_{2kn}
	\end{array}
	\right]
	\in \C^{4kn \times 2kn}, \\
	\mathcal{V}_k &\coloneqq
	\left[
	\begin{array}{c : c}
	I_{2kn} & 
	I_k \otimes
	\begin{bsmallmatrix}
	0 	& -I \\
	I 	& -2ab^{-1} I
	\end{bsmallmatrix} 
	\end{array}
	\right]
	\in \C^{2kn \times 4kn}.
	\end{align}
\eseq
\end{lemma}
\begin{proof}
The factorization follows from the following if-and-only-if equivalences
\begin{align}
	\label{eq:UkVk_alt}
	\mathcal{U}_k\mathcal{V}_k &=
	\left[
	\renewcommand*{\arraystretch}{1.3}
	\begin{array}{c : c}	
	I_{k} \otimes 
	\begin{bsmallmatrix}
	2aI 	& -bI \\
	bI 	& 0 \\
	\end{bsmallmatrix} 
	& 
	-b I_{2kn}
	\\
	\hdashline
	b I_{2kn}
	& 
	I_k \otimes
	\begin{bsmallmatrix}
	0 	& -bI \\
	bI 	& -2aI 
	\end{bsmallmatrix} 
	\end{array}
	\right],  \\
	\nonumber
	&= 
	\left[
	\renewcommand*{\arraystretch}{1.3}
	\begin{array}{c : c}	
	I_{k} \otimes 
	\begin{bsmallmatrix}
	aI 	& -bI \\
	bI 	& -aI \\
	\end{bsmallmatrix} 
	& 
	0
	\\
	\hdashline
	0
	& 
	I_k \otimes
	\begin{bsmallmatrix}
	aI 	& -bI \\
	bI 	& -aI 
	\end{bsmallmatrix} 
	\end{array}
	\right]
	+ 
	\left[
	\begin{array}{c:c}
	aI_{2kn}	& -bI_{2kn} \\
	\hdashline
	bI_{2kn} 	& -a I_{2kn} \\
	\end{array}
	\right], \\
	\nonumber	
	&= 
	I_{2k} \otimes 
	\begin{bsmallmatrix}
	aI 	& -bI \\
	bI 	& -aI 
	\end{bsmallmatrix} 
	+ 
	\begin{bsmallmatrix}
	aI 	& -bI \\
	bI 	& -aI 
	\end{bsmallmatrix} 
	\otimes I_{2k},
\end{align}
while
$
	\mathcal{V}_k \mathcal{U}_k 
	= I_k \otimes \begin{bsmallmatrix} 2aI & -bI \\ bI & 0 \end{bsmallmatrix} 
		+ bI_k \otimes \begin{bsmallmatrix} 0 & -I \\ I & -2ab^{-1}I \end{bsmallmatrix} 
	= 2 I_k \otimes C.
$
\end{proof}

Applying \cref{lem:UkVk} to $\mathcal{A}_2$ with $a \coloneqq 1$, $b \coloneqq \gamma\neq0$, and $k \coloneqq n$, 
$\mathcal{V}_k$ gives the reduced row echelon form of $\mathcal{A}_2$,
and so the rank of $\mathcal{A}_2$ is $2\n^2$.
Thus, is it natural to ask if \eqref{eq:large_eig_ctheta} can also be reduced
to a $2\n^2 \times 2\n^2$ generalized eigenvalue problem, and perhaps even a standard one.
Unfortunately, the presence of nonzero off-diagonal blocks $\pm\gamma I$ in 
$
	\big[
	\begin{smallmatrix} 
		I 		& -\gamma I \\
		\gamma I  & -I 
	\end{smallmatrix}
	\big]
$ from $\mathcal{A}_2$ appears to prevent this.
If one attempts to follow \cite{Gu00,GuMOetal06} and similarly partition $W$ into four $\n \times \n$ blocks, 
multiplying out \eqref{eq:sylv_ctheta_x}
for each block of $W$ results in four equations that all involve $x$. 
In contrast, in \cite{Gu00,GuMOetal06}, eigenvalue $x$ ($\alpha$ in their notation)
only appears in the two corresponding equations for the diagonal blocks of $W$ ($X$ in their notation); 
for the off-diagonal blocks, eigenvalue $x$ (again $\alpha$ in their notation) does not appear in these other two equations, 
since as is seen in \cite[Eq.~3.5]{GuMOetal06}, it ends up being multiplied by zero.
Consequently, the reduction techniques of \cite{Gu00,GuMOetal06} do not seem to be applicable to
\eqref{eq:large_eig_ctheta}.

However, since $\mathcal{A}_2$ is singular, $\mathcal{A}_1- \lambda \mathcal{A}_2$ 
can at least be \emph{numerically} deflated into a smaller pencil 
$\mathcal{\widetilde A}_1- \lambda \mathcal{\widetilde A}_2$ whose spectrum is the set of 
finite eigenvalues of \eqref{eq:large_eig_ctheta}; see the deflation routines of
\cite[Section~4.3]{morBenW20d} and \cite[Algorithm~3, Chapter~3.1]{Koe21},
the former of which is implemented in the \texttt{ml\_ct\_dss\_adtf} routine from MORLAB \cite{morBenW19b}.
Although this iterative deflation technique is cubic work, 
and so also $\bigO(\n^6)$ work for \eqref{eq:large_eig_ctheta},
deflating matrix pencils this way can be faster
than computing their eigenvalues with the QZ algorithm.
Furthermore, for \eqref{eq:large_eig_ctheta}, deflation results in a matrix pencil of half the order, 
since it removes $2\n^2$ infinite eigenvalues.
As a result, deflating and then computing the eigenvalues of the resulting pencil 
may be even faster than computing the eigenvalues of \eqref{eq:large_eig_ctheta} directly.
That matrix $\mathcal{\widetilde A}_2$ in the reduced pencil is nonsingular also provides a second important benefit.
Although we must compute the real eigenvalues of \eqref{eq:large_eig_ctheta},
we cannot expect its real eigenvalues to be exactly real in the presence of rounding errors.
The key question then is how far away from the real axis can a computed eigenvalue be allowed to be
while still being \emph{deemed} a real eigenvalue of \eqref{eq:large_eig_ctheta}?
For the reduced pencil, a reliable tolerance is  
$\texttt{tol} \cdot \epsilon_\mathrm{mach} \cdot \| \mathcal{\widetilde A}_2^{-1} \mathcal{\widetilde A}_1 \|_\infty$,
where $\texttt{tol} > 1$ is provided by the user and $\epsilon_\mathrm{mach}$ is the machine precision.
Note that we still recommend computing the eigenvalues of $\mathcal{\widetilde A}_1- \lambda \mathcal{\widetilde A}_2$ as a generalized eigenvalue problem, instead of using the matrix $\mathcal{\widetilde A}_2^{-1} \mathcal{\widetilde A}_1$,
as we have observed that the condition number of $\mathcal{\widetilde A}_2$ is generally very large in practice.

\subsection{Faster computation of the real eigenvalues of $\mathcal{A}_1w = x \mathcal{A}_2w$}
\label{sec:fast_eig_cont}
We now show how the $\bigO(\n^6)$ theoretical work complexity of \cref{alg:owr_bt}
can be reduced.
To do this, we will work with the matrices in \eqref{eq:large_eig_ctheta} and
adapt the divide-and-conquer approach proposed in \cite[Section~3.3.2]{GuMOetal06}
for faster computation of $\dtu$.
At a high level, this efficiency improvement relies on two principles.
First, although the matrices in \eqref{eq:large_eig_ctheta} are $4\n^2 \times 4\n^2$ in size,
they arose from vectorizing the two corresponding $2\n \times 2\n$ Sylvester forms in \eqref{eq:sylv_ctheta_x}.
As such, applying $\mathcal{A}_1$ and $\mathcal{A}_2$, or their inverses to a vector
can actually be done with just $\bigO(\n^3)$ work.
In turn, this means that for any shift $\shift \in \C$, $(\mathcal{A}_1 - \shift \mathcal{A}_2)^{-1}$
can be applied to a vector with $\bigO(\n^3)$ work (we will clarify how these computations
are done in a moment).
Consequently, a shift-and-invert eigenvalue solver, e.g., \texttt{eigs} in \matlab, 
can be employed to find the eigenvalues of \eqref{eq:large_eig_ctheta} 
that are closest to a shift $\shift$ with only $\bigO(n^3)$ work.
Second, given a matrix $\mathcal{X} \in \C^{q \times q}$, suppose one only wants its eigenvalues that are along a line segment, say, an interval $[0,D]$ on the real axis for some $D > 0$.
Then the recursive iteration given by \cite[Algorithm~4]{GuMOetal06},
which uses a shift-and-invert eigensolver,
can locate all eigenvalues of $\mathcal{X}$ in $[0,D]$ with at most $2q + 1$ shifts in the worst case
and $\bigO(\sqrt{q})$ shifts if the eigenvalues of $\mathcal{X}$ are distributed uniformly.
For brevity, we forgo the details of describing \cite[Algorithm~4]{GuMOetal06},
but note that the cost of choosing $D$ unnecessarily large only results in about four extra shifts \cite[p.~490]{GuMOetal06}.
Thus, the theory says that by adapting this divide-and-conquer technique to compute 
the positive real eigenvalues of \eqref{eq:large_eig_ctheta}, 
the overall work complexity of \cref{alg:owr_bt} will be reduced to 
$\bigO(\n^4)$ work on average and $\bigO(\n^5)$ in the worst case.

We now explain how $\mathcal{A}_1$, $\mathcal{A}_2$, and $(\mathcal{A}_1 - \shift \mathcal{A}_2)^{-1}$
can all be applied to a vector $w \in \C^{4\n^2}$ with at most $\bigO(\n^3)$ work.
As $\mathcal{A}_2$ has only $10\n^2$ nonzero entries, it suffices to store it in a sparse matrix format.
Computing $\mathcal{A}_1 w$
can be done efficiently via vectorizing the first Sylvester form in \eqref{eq:sylv_ctheta_x},
i.e., 
$ 
	\vecop
	\left(
	\begin{bsmallmatrix} 
		A 	& 0 \\ 
		0 	& -A^*
	\end{bsmallmatrix}
	W + W
	\begin{bsmallmatrix} 
		A^* - \eta \emit I  		& -\gamma \eta \costh I \\
		\gamma \eta \costh I 		& \eta \eit I - A
	\end{bsmallmatrix}
	\right),
$ 
where $W \in \C^{2\n \times 2\n}$ and $w = \vecop(W)$.
The dominant cost in obtaining $\mathcal{A}_1 w$ is
the two matrix multiplies with $W$, hence it too can be done in $\bigO(\n^3)$ work.
For $y \in \C^{4n^2}$, we can efficiently obtain $w = (\mathcal{A}_1 - \shift \mathcal{A}_2)^{-1}y$ 
by considering $(\mathcal{A}_1 - \shift \mathcal{A}_2) w = y$.
This ``unvectorizes" into \eqref{eq:sylv_ctheta}, provided that $x$ is replaced with $s$
and the zero in its right-hand side is replaced by $Y$, where $y = \vecop(Y)$.
By solving the resulting Sylvester equation and vectorizing its solution $W$, 
$w = (\mathcal{A}_1 - \shift \mathcal{A}_2)^{-1}y$ is computed in $\bigO(\n^3)$ work.

\section{Continuous-time $\Kcon$ algorithms based on variable-distance pairs}
\label{sec:alternatives}
Having developed the first globally convergent iteration for continuous-time Kreiss constants,
we now develop two more, namely \cref{alg:owr,alg:tri}, which can be considered closer analogues of the 
two $\dtu$ methods of \cite{BurLO04} described in \S\ref{sec:dtu}.
As previously  discussed  at the end of \S\ref{sec:analogues}, 
the lower bounds provided by \cref{cor:kinv_cv,cor:kinv_ch}, due to their dependency 
on the unknown $x_\star$, prevent creating direct $\Kcon$ analogues of the 
trisection and optimization-with-restarts of \cite{BurLO04}.  
However, these lower bounds are the result of the assumption of looking for 
pairs of level-set points that are a \emph{fixed distance} $\eta$ apart,
as Gu originally used for \cref{thm:dtu}.
As we are about to show, if we change this assumption, 
then we can obtain \emph{different} lower bound results
than those given in \eqref{eq:kinv_cv_lb} and \eqref{eq:kinv_ch_lb}.

\begin{theorem}
\label{thm:kinv_cv_vari}
For $A \in \C^{\n \times \n}$ with $\alpha(A) < 0$,
let $\gamma \in [0,1)$,
$\eta \geq 0$, and $(x_\star,y_\star)$ be a global minimizer of \eqref{eq:kinv_cont}.
If $\Kinv \leq \gamma$ and $\eta \in [0,2(\gamma - \Kinv)]$, then 
there exists a pair $x,y \in \R$ with $x > 0 $ such that
\beq
	\label{eq:kinv_cv_vari_ub}
	g(x,y) = g(x,y+x\eta) = \gamma.
\eeq
\end{theorem}

\begin{corollary}
\label{cor:kinv_cv_vari}
For $A \in \C^{\n \times \n}$ with $\alpha(A) < 0$, let $\gamma \in [0,1)$, $\eta \geq 0$,
and $(x_\star,y_\star)$ be a global minimizer of \eqref{eq:kinv_cont}.
If there do not exist any pairs $x,y \in \R$ with $x > 0$ such that \eqref{eq:kinv_cv_vari_ub} holds,
then 
\beq
	\label{eq:kinv_cv_vari_lb}
	\Kinv > \gamma - \tfrac{\eta}{2x_\star}.
\eeq
\end{corollary}

\begin{proof}[Proof of \cref{thm:kinv_cv_vari}]
The proof follows similarly to the proof of \cref{thm:kinv_cv} with the following 
modification.
By using the variable distance $x\eta$, for $\gamma \in [\Kinv,1)$
we instead obtain that $\eta \in [0,2(\gamma - \Kinv)]$ is a sufficient condition for \eqref{eq:kinv_cv_vari_ub} to hold;
this is because this choice cancels out the $x_\star$ in the proof, as $|y_1 - y_2| = x_\star \eta $.
\end{proof}

Thus, by looking for vertical pairs of points that are this particular \emph{variable distance} apart, i.e., $x \eta$,
a corresponding certificate procedure would either assert that $\gamma \geq \Kinv$ holds or that \eqref{eq:kinv_cv_vari_lb} does.  Such a certificate would avoid the need for the backtracking procedure that was necessary for \cref{alg:owr_bt}.
While this might seem to be obviously preferable to \cref{alg:owr_bt}, as a bit of foreshadowing,
we note that the large eigenvalue problem that results for this variable-distance certificate 
is quite different from \eqref{eq:large_eig_ctheta} and has its own downsides.
We now describe this certificate to complete \cref{alg:owr,alg:tri}.

\begin{algfloat}[t]
\begin{algorithm}[H]
\floatname{algorithm}{Algorithm}
\caption{Optimization-with-restarts (no backtracking)}
\label{alg:owr}
\begin{algorithmic}[1]
	\REQUIRE{  
		A nonnormal matrix $A \in \C^{n \times n}$ with $\alpha(A) < 0$,
		$x_0 > 0$ and $y_0 \in \R$ such that $g(x_0,y_0) < 1$,
		and a tolerance $\gamma_\mathrm{tol} > 0$.
		}
	\ENSURE{ 
		$g_k^{-1} \approx \Kcon$ (continuous-time).
		\\ \quad
	}
	
	\WHILE { true } 
		\STATE $(\hat x, \hat y) \gets $ computed local/global minimizer of \eqref{eq:kinv_cont} initialized from $(x_0,y_0)$
		\STATE $g_k \gets g(\hat x, \hat y)$
		\STATE $\gamma \gets g_k (1 - 0.5\cdot \gamma_\mathrm{tol})$
		\STATE $\eta \gets g_k \cdot \gamma_\mathrm{tol}$
	
		\STATE Perform the 2D level-set test of \S\ref{sec:2d_test_cv_vari} with current $\gamma$ and $\eta$
		\IF { test finds \emph{any} level-set points  }
			\STATE $(x_0,y_0) \gets$ one of these points
				\COMMENT{Goto line 2 to restart optimization.}
		\ELSE
			\RETURN \COMMENT{ $\Kinv > g_k \cdot (1 - \gamma_\mathrm{tol})$ holds.}
		\ENDIF
	\ENDWHILE
\end{algorithmic}
\end{algorithm}
\end{algfloat}

\begin{algfloat}[t]
\begin{algorithm}[H]
\floatname{algorithm}{Algorithm}
\caption{Trisection}
\label{alg:tri}
\begin{algorithmic}[1]
	\REQUIRE{  
		A nonnormal matrix $A \in \C^{n \times n}$ with $\alpha(A) < 0$
		and a tolerance $\gamma_\mathrm{tol} > 0$.
		}
	\ENSURE{ 
		$\texttt{ub}^{-1} \approx \Kcon$ (continuous-time).
		\\ \quad
	}
	
	\STATE $\texttt{lb} \gets 0$
	\STATE $\texttt{ub} \gets g(x_0,y_0)$ for some $x_0 > 0$ and $y_0 \in \R$
	\WHILE { $(\texttt{ub} - \texttt{lb}) > \texttt{ub} \cdot \gamma_\mathrm{tol}$ } 
		\STATE $\texttt{diff} \gets \texttt{ub} - \texttt{lb}$
		\STATE $\eta \gets \tfrac{2}{3} \cdot \texttt{diff}$
		\STATE $\gamma \gets \texttt{lb} + \eta$
		\STATE Perform the 2D level-set test of \S\ref{sec:2d_test_cv_vari} with current $\gamma$ and $\eta$
		\IF { test finds \emph{any} level-set points } 
			\STATE $\texttt{ub} = \gamma$
		\ELSE
			\STATE $\texttt{lb} = \texttt{lb} + \tfrac{1}{3} \cdot \texttt{diff}$	
		\ENDIF
	\ENDWHILE
\end{algorithmic}
\end{algorithm}
\algnote{
While \cite[p.~358]{BurLO04} states that their trisection-based $\dtu$ algorithm 
converges ``to any prescribed absolute accuracy",
any desired relative accuracy can be obtained by simply choosing a stopping condition like
the one we have used here in line 3. 
}
\end{algfloat}

\subsection{A 2D level-set test for variable-distance pairs}
\label{sec:2d_test_cv_vari}
Suppose $\gamma$ is both a singular value of $G(x,y)$ and $G(x,y+x\eta)$, with respective left and right singular vectors pairs $u$,$v$ and $\hat u$,$\hat v$.
Via \cref{thm:1d_vert} for $G(x,y)$ and a similar argument as its proof for $G(x,y + x\eta)$,  
we have the following two Hamiltonian eigenvalue problems
\bseq
\label{eq:eig_pair_cv_vari}
\begin{align}
	\begin{bmatrix}
		A - xI 		& \gamma xI 	\\
		-\gamma xI  	& xI -A^*  		\\
	\end{bmatrix}
	\begin{bmatrix} v \\ u \end{bmatrix}
	&= \imagunit y \begin{bmatrix} v \\ u \end{bmatrix}, \\
	\begin{bmatrix}
		A - x(1+\imagunit \eta)I 	& \gamma x I 			 \\
		-\gamma x I 			& x(1-\imagunit\eta)I - A^*  \\
	\end{bmatrix}
	\begin{bmatrix} \hat v \\ \hat u \end{bmatrix}
	&= \imagunit y \begin{bmatrix} \hat v \\ \hat u \end{bmatrix}.
\end{align}
\eseq
These two Hamiltonian matrices have an eigenvalue in common if
\beq
	\label{eq:sylv_cv_vari}
	\begin{bmatrix}
		A - xI 		& \gamma xI 	\\
		-\gamma xI  	& xI -A^*  		\\
	\end{bmatrix}
	W + W
	\begin{bmatrix}
		A^* - x(1 - \imagunit \eta)I 	& -\gamma x I \\
		\gamma x I 			& x(1 + \imagunit\eta)I - A \\
	\end{bmatrix}
	= 0
\eeq
has a nonzero solution $W \in \C^{2\n\times2\n}$.
Separating out the terms involving $x$, we have
\begin{multline}
	\label{eq:sylv_cv_vari_x}
	\left(
	\begin{bmatrix} 
		A 	& 0 		\\ 
		0 	& -A^* 	\\
	\end{bmatrix}
	W + W
	\begin{bmatrix}
	 	A^*	& 0 	\\
		0 	& -A 	\\
	\end{bmatrix} 
	\right) \\
	- x 
	\left( 
	\begin{bmatrix} 
		I 		& -\gamma I 	\\ 
		\gamma I  & -I 			\\ 	
	\end{bmatrix} 
	W + W
	\begin{bmatrix} 
		(1 - \imagunit \eta) I 		& \gamma I \\ 
		-\gamma I  			& -(1 + \imagunit \eta)I 
	\end{bmatrix} 
	\right) = 0.
\end{multline}
Rewriting both Sylvester forms using the vectorize operator, 
and letting $w = \vecop(W)$, we have
the following generalized eigenvalue problem
\begin{align}
	\label{eq:large_eig_cv_vari}
	\mathcal{B}_1 w &= x \mathcal{B}_2 w, 
	\quad \text{where} \\
	\nonumber
	\mathcal{B}_1 &=  
	I_{2n} \otimes 
		\begin{bmatrix} 
		A 	& 0 		\\ 
		0 	& -A^*
	\end{bmatrix}
	+
	\begin{bmatrix} 
		\,\overline{A} 	& 0 \\
		0 			& -A\tp
	\end{bmatrix}
	\otimes I_{2n}, \\
	\nonumber
	\mathcal{B}_2 &= 
	I_{2n} \otimes 
	\begin{bmatrix} 
		I 			& -\gamma I \\ 
		\gamma I  	& -I 
	\end{bmatrix} 
	+
	\begin{bmatrix} 
		(1 - \imagunit \eta) I 	& -\gamma I \\ 
		\gamma I  		& -(1 + \imagunit \eta)I 
	\end{bmatrix} 
	\otimes I_{2n}.
\end{align}

To check if \eqref{eq:kinv_cv_vari_lb} holds with this variable-distance certificate,
we proceed similarly to our fixed-distance certificate from \S\ref{sec:2d_test_fixed}.
Thus, after initially computing all the positive real eigenvalues of \eqref{eq:large_eig_cv_vari},
we then apply \cref{thm:1d_vert} to each of these candidate vertical lines to see
if we detect any points on a $\gamma$-level set (or lower) of $g(x,y)$.

\begin{keyremark}
\label{rem:improved_test}
We now explain how the designs of our 2D level-set tests intentionally deviate
from the $\dtu$ procedure Gu proposed  in \cite[pp.~995--997]{Gu00} 
and how this results in much better reliability.
If one were to more closely follow Gu's procedure as it is written, for each positive real eigenvalue $\hat x > 0$ 
of \eqref{eq:large_eig_cv_vari}, one would instead compute the eigenvalues of
the two matrices in \eqref{eq:eig_pair_cv_vari} and then check whether these two spectra
have any imaginary eigenvalues in common.
If a shared imaginary eigenvalue is detected, then a pair of level-set points a distance $x\eta$ apart 
has been detected\footnote{Note that while  
the bottom of \cite[p.~996]{Gu00} says that detected points
would be on the $\gamma$-level set, 
technically that only holds if $\gamma$ is also the minimum singular value at both of these points, which may or may not be true.}
and $\gamma \geq \Kinv$ must hold.
Otherwise, if none of the pairs of eigenvalue problems from \eqref{eq:eig_pair_cv_vari} share imaginary eigenvalues,
then following Gu would mean asserting that lower bound \eqref{eq:kinv_cv_vari_lb} must hold.
While Gu's procedure is sound in exact arithmetic, 
trying to assert whether two matrices share an (imaginary) eigenvalue
is exceptionally difficult to do reliably in the presence of rounding errors.
Furthermore, this matching also assumes that the real eigenvalues $\hat x$ have been computed accurately enough
such that the matrices \eqref{eq:eig_pair_cv_vari} would indeed share an 
eigenvalue $\imagunit \hat y$, assuming $(\hat x,\hat y)$ satisfies \eqref{eq:kinv_cv_vari_ub}, 
which is another source of numerical uncertainty; see \cite[Section~5]{GuMOetal06}.
Suppose that there are indeed points a distance $x \eta$ apart vertically on the $\gamma$-level set of $g(x,y)$,
but rounding errors prevent their detection. 
In this case, asserting that \eqref{eq:kinv_cv_vari_lb} holds may be erroneous.
This can be a critical failure because the lower bound is only true \emph{if no such pair of level-set points exist}, 
not if the procedure fails to detect them due to numerical problems!
This is a major reason why Gu's procedure can have such numerical difficulties, particularly when $\eta$ is small.
Though \cite{GuMOetal06} improves Gu's procedure to make it faster, it too follows the same the idea of 
checking whether or not two matrices share imaginary eigenvalues and thus also inherits these numerical problems.
However, in the course of our work here, we have realized that 
Gu's $\dtu$ procedure ironically follows \cref{thm:dtu,cor:dtu} too closely,
because as it turns outs, checking whether or not level-set points 
are a distance $x\eta$ (or $\eta$) apart is \emph{entirely unnecessary}.
While indeed \eqref{eq:kinv_cv_vari_lb} must hold if no points satisfy \eqref{eq:kinv_cv_vari_ub},
another sufficient condition for \eqref{eq:kinv_cv_vari_lb} to hold is that 
the procedure itself does not generate \emph{any} level-set points whatsoever, on the $\gamma$-level set or lower and/or as pairs or single points.
If any level-set points are detected, clearly $\gamma \geq \Kinv$ holds, while none being generated
implies \eqref{eq:kinv_cv_vari_ub} cannot hold, which in turn asserts \eqref{eq:kinv_cv_vari_lb} must hold.
Note that our new way of performing these 2D level-set tests is not only a theoretical improvement; 
we initially designed our tests to check for common imaginary eigenvalues, like Gu's procedure,
but discovered that the aforementioned numerical issues were impeding the reliability of the codes.
Applying our modifications to the $\dtu$ algorithms of \cite{Gu00,BurLO04,GuMOetal06}
should improve their reliability as well.
\end{keyremark}

\subsection{Properties of the eigensystem $\mathcal{B}_1 w = x \mathcal{B}_2w$ and its solution}
Unlike \eqref{eq:large_eig_ctheta}, which can at least be numerically deflated
to an order $2\n^2$ generalized eigenvalue problem,
we now show that \eqref{eq:large_eig_cv_vari} cannot be similarly reduced, as 
$\mathcal{B}_2$ is generally nonsingular.  In fact, $\mathcal{B}_2$ has an explicit inverse.

\begin{theorem}
\label{thm:kronsuminv}
Let $C \coloneqq \begin{bsmallmatrix} aI & -bI \\ bI & -aI\end{bsmallmatrix} \in \C^{2n \times 2n}$ with $a,b \in \C$ and $b \neq 0$, and let $\mathcal{M} \coloneqq I_{2k} \otimes C + C \otimes I_{2k}$.
For $k \in \N$ and $s \in \C$, 
define matrix $\mathcal{D} \coloneqq s I_{4kn} + \mathcal{M} $
and scalar \mbox{$\beta \coloneqq s^2  + 4(b^2 - a^2)$}.
Then $\mathcal{D}$ is invertible with inverse 
\beq
	\label{eq:Dinv}
	\mathcal{D}^{-1} = 
	s^{-1} I_{4kn} - \beta^{-1} \mathcal{U}_k 
	\left(
	I_{2kn} - 2s^{-1} I_k \otimes C
	\right)
	\mathcal{V}_k
\eeq
if and only if $s$ and $\beta$ are both nonzero,
where $\mathcal{U}_k$ and $\mathcal{V}_k$ are defined in \eqref{eq:UkVk}.
Moreover, if $k=n$, then the following simpler formula holds
\beq
	\label{eq:Dinv_simple}
	\mathcal{D}^{-1} = 
	s^{-1} I_{4n^2} - \beta^{-1} \mathcal{M} + (s\beta)^{-1}\mathcal{M}^2.
\eeq
\end{theorem}
\begin{proof}
Via \cref{lem:UkVk} applied to $\mathcal{M}$, we have that $\mathcal{D} = sI_{4kn} + \mathcal{U}_k\mathcal{V}_k$.
As the eigenvalues of $C$ are $\pm\sqrt{a^2 - b^2}$, by \cref{thm:kronsum},
the eigenvalues of $\mathcal{M}$ are $\pm2\sqrt{a^2 - b^2}$ and zero.
Thus, $\mathcal{D}$ is invertible if and only if $-s$ is not equal to any of these eigenvalues,
which is equivalent to $s \neq 0$ and $\beta \neq 0$.
For \eqref{eq:Dinv},
we apply the Sherman-Morrison-Woodbury formula to 
$(sI_{4kn} + \mathcal{U}_k\mathcal{V}_k)^{-1}$
and use $\mathcal{V}_k\mathcal{U}_k = 2I_k \otimes C$ from \cref{lem:UkVk}, yielding
\begin{align*}
	\mathcal{D}^{-1} 
	&= s^{-1} I_{4kn}  - s^{-1} \mathcal{U}_k \left(I_{2kn} + s^{-1} \mathcal{V}_k\mathcal{U}_k \right)^{-1} s^{-1} \mathcal{V}_k, \\
	&= s^{-1} I_{4kn}  - s^{-1} \mathcal{U}_k \left(sI_{2kn} +  2I_k \otimes C \right)^{-1} \mathcal{V}_k, \\
	&= s^{-1} I_{4kn}  - s^{-1} \mathcal{U}_k \left(I_k \otimes (sI_{2n} +  2C) \right)^{-1} \mathcal{V}_k, \\
	&= s^{-1} I_{4kn}  - s^{-1} \mathcal{U}_k \left(I_k \otimes (sI_{2n} +  2C)^{-1} \right) \mathcal{V}_k, \\
	&= s^{-1} I_{4kn}  - s^{-1} \mathcal{U}_k \left(
		I_k \otimes \tfrac{1}{s^2 + 4(b^2 - a^2)}\begin{bsmallmatrix} (s - 2a)I & 2bI \\ -2bI & (s+2a)I \end{bsmallmatrix}
		\right) \mathcal{V}_k,	 \\
	&= s^{-1} I_{4kn}  - s^{-1}\beta^{-1} \mathcal{U}_k \left(
		s I_{2kn} - 2 I_k \otimes C \right) \mathcal{V}_k.	 		
\end{align*}
For \eqref{eq:Dinv_simple}, now with $k=n$, first note that $C^2 = \phi I_{2n}$ with $\phi = (a^2-b^2)$, while
\begin{align*}
	\mathcal{M}^2 
	&= \left(I_{2n} \otimes C^2 + C^2 \otimes I_{2n}\right) + \left(C \otimes C + C \otimes C\right) 
	= 2 \left(\phi I_{4n^2} + C\otimes C\right), \\
	\mathcal{M}^3 
	&= 2\left(\phi I_{4n^2} + C\otimes C\right) \mathcal{M}
	= 2\phi \mathcal{M} + 2 \left(C\otimes C^2 + C^2 \otimes C\right) = 4\phi \mathcal{M},
\end{align*}
where we have used the mixed-product property of $\otimes$.
Then by multiplying \eqref{eq:Dinv_simple} with $sI_{4n^2} + \mathcal{M}$ and noting that
$\beta - s^2 + 4\phi = 0$, it follows that
\begin{align*}
	I_{4n^2} + s^{-1}\mathcal{M} - s\beta^{-1} \mathcal{M} + \left(s\beta\right)^{-1}\mathcal{M}^3
	&= I_{4n^2} + \left(s^{-1} - s\beta^{-1} + 4\phi \left(s\beta\right)^{-1}\right) \mathcal{M} \\
	&= I_{4n^2} + \left(s\beta\right)^{-1} \left(\beta - s^2 + 4\phi \right) \mathcal{M} = I_{4n^2}.
\end{align*}
\end{proof}

Applying \cref{thm:kronsuminv} to $\mathcal{B}_2$, with
$a \coloneqq 1$, $b \coloneqq \gamma \neq 0$, $s \coloneqq -\imagunit \eta \neq 0$,
we see that $\beta \neq 0$ holds if $\eta \neq \pm 2 \sqrt{\gamma^ 2 - 1}$ $(\not\in \R$ if $|\gamma| < 1$),
hence $\mathcal{B}_2$ is generically invertible with 
\beq
	\label{eq:B2_inv}
	\mathcal{B}_2^{-1} = -(\imagunit \eta)^{-1} I_{4n^2} - \beta^{-1} \mathcal{B}_2 - (\imagunit \eta \beta)^{-1}\mathcal{B}^2,
\eeq
where $\beta = 4(\gamma^2 - 1) - \eta^2$.
On the upside, one could thus perform our variable-distance certificate
by computing the eigenvalues of $\mathcal{B}_2^{-1}\mathcal{B}_1$,
while our fixed-distance certificate requires solving 
the general eigenvalue problem \eqref{eq:large_eig_ctheta}.
However, using $\mathcal{B}_2^{-1}\mathcal{B}_1$ can introduce numerical issues
and the order of this eigenvalue problem is $4\n^2$,
whereas for the fixed-distance case, we could instead solve the numerically deflated order $2\n^2$ problem 
discussed in \S\ref{sec:2d_test_fixed_props}.
Furthermore, from \eqref{eq:B2_inv}, it is clear that $\|\mathcal{B}_2^{-1}\| \to \infty$ as $\eta \to 0$; in the presence of rounding errors, this may make it quite difficult to reliably 
ascertain which eigenvalues of $\mathcal{B}_2^{-1}\mathcal{B}_1$ should be considered 
real-valued.

At this point, one might ask if it would have been better to consider
a variable-distance certificate using a horizontal orientation instead of a vertical one.
We consider this in \cref{apdx:2d_test_ch_vari} and note here that it results in 
the matrix pencil $\mathcal{B}_1 - \lambda \mathcal{\widetilde B}_2$, where,
like $\mathcal{B}_2$, the matrix $\mathcal{\widetilde B}_2$ is nonsingular but becomes singular as $\eta\to0$. 

\subsection{Adapting divide-and-conquer for $\mathcal{B}_1w = x \mathcal{B}_2w$}
\label{sec:fast_eig_cont_vari}
We now describe the computations 
needed for a divide-and-conquer version of our variable-distance certificate.
Note that the numerical reliability of this approach might be different than our fixed-distance
certificate, as \eqref{eq:large_eig_cv_vari} only has finite eigenvalues, while 
\eqref{eq:large_eig_ctheta} has finite and infinite eigenvalues.
For $w \in \C^{4\n^2}$, $\mathcal{B}_2 w$ can be efficiently obtained
by storing $\mathcal{B}_2$ in a sparse format.
Computing $\mathcal{B}_1 w$
is $\bigO(\n^3)$ work, which is done by vectorizing the first matrix in \eqref{eq:sylv_cv_vari_x},
i.e.,
$
	\vecop
	\left(
	\begin{bsmallmatrix} 
		A 	& 0 \\ 
		0 	& -A^*
	\end{bsmallmatrix}
	W + W
	\begin{bsmallmatrix} 
		A^* 	& 0 \\ 
		0 	& -A
	\end{bsmallmatrix}
	\right)
$, 
where $W \in \C^{2\n \times 2\n}$ and $w = \vecop(W)$.
To obtain $w = (\mathcal{B}_1 - \shift \mathcal{B}_2)^{-1}y$ for $y \in \C^{4\n^2}$, 
consider $(\mathcal{B}_1 - \shift \mathcal{B}_2) w = y$.
This  ``unvectorizes" into \eqref{eq:sylv_cv_vari} provided that $x$ is replaced by $s$
and the zero on its right-hand side is replaced by $Y$, where $y = \vecop(Y)$.
Vectorizing the solution of the resulting Sylvester equation yields 
$w = (\mathcal{B}_1 - \shift \mathcal{B}_2)^{-1}y$ in $\bigO(\n^3)$ work.

\section{Algorithms for discrete-time $\Kcon$}
\label{sec:disc}
To adapt \cref{alg:owr_bt,alg:owr,alg:tri} to compute discrete-time Kreiss constants,
we need to develop discrete-time versions of the 2D level-set tests from \S\ref{sec:2d_test_fixed}
and \S\ref{sec:2d_test_cv_vari};
the other necessary components for the two optimization-with-restarts algorithms have already been 
discussed in \S\ref{sec:opt_disc}.
We begin by developing discrete-time analogues of \cref{thm:kinv_cv,cor:kinv_cv}.
For the discrete-time case, we now additionally assume that $0 \not\in \Lambda(A)$.

\begin{theorem}
\label{thm:kinv_dr}
For $A \in \C^{\n \times \n}$ with $\rho(A) < 1$,
let $\gamma \in [0,1)$, $\eta \geq 0$, 
and $(r_\star,\theta_\star)$ be a global minimizer of \eqref{eq:kinv_disc}.
If $\Kinv \leq \gamma$ and $\eta \in \big(0,\tfrac{2(r_\star - 1)}{1 + \gamma}(\gamma - \Kinv)\big]$, then 
there exists an $r > 1$ and $\theta \in [0,2\pi)$ such that
\beq
	\label{eq:kinv_dr_ub}
	h(r,\theta) = h(r + \eta,\theta) = \gamma.
\eeq
\end{theorem}

\begin{corollary}
\label{cor:kinv_dr}
For $A \in \C^{\n \times \n}$ with $\rho(A) < 1$, let $\gamma \in [0,1)$, $\eta \geq 0$,
and $(r_\star,\theta_\star)$ be a global minimizer of \eqref{eq:kinv_disc}.
If there do not exist any pairs $r,\theta \in \R$ with $r > 1$ such that \eqref{eq:kinv_dr_ub} holds,
then 
\beq
	\label{eq:kinv_dr_lb}
	\Kinv > \gamma - \tfrac{\eta(1 + \gamma)}{2(r_\star - 1)}.
\eeq
\end{corollary}

\begin{proof}[Proof of \cref{thm:kinv_dr}]
If $\gamma=\Kinv$, the proof is trivially satisfied with $\eta = 0$, so assume that $\gamma \in (\Kinv,1)$.
Since $\rho(A) < 1$, it follows that $\lim_{r\to1^+} h(r,\theta) = \infty$ for any $\theta \in \R$,
and so $h(r_\star,\theta_\star) = \Kinv$ with $r_\star > 1$.
Now consider the strict lower level set $\mathcal{L}_\gamma \coloneqq \{ (r,\theta) : h(r,\theta) < \gamma, r > 1\}$,
which is clearly open and also bounded; see \cite[Theorem~3.2]{Mit19a}.
Let $\mathcal{L}$ be the (open) connected component of $\mathcal{L}_\gamma$ such that
$(r_\star,\theta_\star) \in \mathcal{L}$ and let $\mathcal{G} \coloneqq \mathcal{L}^\mathrm{H}$, i.e., the 
simply connected hull of $\mathcal{L}$.

By continuity of $h(r,\theta)$ and the boundedness of $\mathcal{G}$,
there must exist $b_1,b_2 \in \bd \mathcal{G}$
\[
	b_1 = (r_\star - \eta_1, \theta_\star) 
	\qquad \text{and} \qquad
 	b_2 = (r_\star + \eta_2, \theta_\star)
\]
such that 
\beq
	\label{eq:shifted_disc}
	h(r_\star  - \eta_1, \theta_\star) = h(r_\star + \eta_2, \theta_\star) = \gamma,
\eeq
where $\eta_1,\eta_2 > 0$.
Furthermore, we can assume that $\eta_1$ and $\eta_2$ are the smallest positive values 
such that \eqref{eq:shifted_disc} holds with $b_1,b_2 \in \bd \mathcal{G}$.
Defining
\beq
	\label{eq:h_numerator}
	\tilde h(r) \coloneqq
	h(r + 1,\theta) = \frac{\smin\left((r+1)\eit I - A\right)}{r},
\eeq
$\tilde h(r)$ has a global minimizer $r_\star - 1$ on domain $(0,\infty)$ and its numerator has a GLC~of~1.
Applying \cref{lem:foverx} to $q(r) \coloneqq \tilde h(r)$ with $a = r_\star - \eta_1 - 1$ and $b = r_\star +\eta_2 - 1$ yields
\[
	\eta_1 + \eta_2 \geq \tfrac{2(r_\star - 1)}{1 + \gamma}(\gamma - \Kinv).
\]

Now suppose that $\eta \in \big(0,\tfrac{2(r_\star - 1)}{1 + \gamma}(\gamma - \Kinv)\big]$ so $\eta \leq \eta_1 + \eta_2$.
If $\eta = \eta_1 + \eta_2$, \eqref{eq:kinv_dr_ub} is clearly satisfied, 
so instead assume that $\eta < \eta_1 + \eta_2$.  Considering the set
$\mathcal{\widehat G} \coloneqq \{ (r - \eta,\theta) : (r,\theta) \in \mathcal{G}\}$,
without loss of generality, we can assume $\theta_\star = 0$, and so $\mathcal{\widehat G}$
is simply $\mathcal{G}$ shifted left by the amount $\eta$.
Thus, as in the end of the proof of \cref{thm:kinv_ch},
the rest of the argument follows analogously to the last paragraph of the proof of \cref{thm:kinv_cv}.
\end{proof}

\subsection{A 1D circular level-set test}
For all of our discrete-time $\Kcon$ algorithms, we will also need the following theorem.
\begin{theorem}
\label{thm:1d_circle}
Given $\gamma, r,\theta \in \R$, with $\gamma \geq 0$ and $r \neq 1$,
$\gamma$ is a singular value of $H(r,\theta)$ defined in \eqref{eq:h_disc} if and only if $\eit$ is an eigenvalue of
the symplectic matrix pencil
\beq
	\label{eq:symp_pen}
	\begin{bmatrix} 
		A    	& \gamma(r-1)I \\
		0 	& rI
	\end{bmatrix}
	- \lambda
	\begin{bmatrix} 
		rI    			& 0 	\\
		\gamma(r-1)I 	& A^*
	\end{bmatrix}.
\eeq
Furthermore, if $A$ is invertible and $r \neq 0$, zero is not an eigenvalue of \eqref{eq:symp_pen}
and the matrix pencil is regular.
\end{theorem}
\begin{proof}
It is easy to verify \eqref{eq:symp_pen} is symplectic, and under the additional assumptions,
also regular and that zero cannot be an eigenvalue.
Now suppose $\gamma$ is a singular value  $H(r,\theta)$ with left and right singular vectors $u$ and $v$,
and so
\[
\thinmuskip=2.4mu
\medmuskip=3.4mu
\thickmuskip=4.4mu
	\gamma 
	\begin{bmatrix} u \\ v \end{bmatrix}
	=
	\begin{bmatrix} 
		H(r,\theta) & 0 	\\
		0 		& H(r,\theta)^* 
	\end{bmatrix}
	\begin{bmatrix} v \\ u \end{bmatrix}
	\ \Leftrightarrow \
	\gamma (r-1)
	\begin{bmatrix} u \\ -\eit v \end{bmatrix}
	=
	\begin{bmatrix} 
		r\eit I - A 	& 0\\
		0				& \eit A^* - r I
	\end{bmatrix}
	\begin{bmatrix} v \\ u \end{bmatrix}.	
\]
Rearranging terms, this is equivalent to
\[
	\gamma (r-1)
	\begin{bmatrix} u \\ 0 \end{bmatrix}
	+
	\begin{bmatrix} 
		A 	& 0\\
		0	& r I
	\end{bmatrix}
	\begin{bmatrix} v \\ u \end{bmatrix}
	=
	\eit \gamma(r-1)
	\begin{bmatrix} 0 \\  v \end{bmatrix}
	+
	\eit
	\begin{bmatrix} 
		r I 	& 0\\
		0	& A^* 
	\end{bmatrix}
	\begin{bmatrix} v \\ u \end{bmatrix}.	
\]
\end{proof}

Similar to \cref{thm:1d_vert}, note that the unimodular eigenvalues of \eqref{eq:symp_pen}
correspond to points that are either on the $\gamma$-level set of $h(r,\theta)$ or on \emph{lower} level sets.
The structure-preserving eigensolvers of \cite{BenMX98a,BenMX98b} 
can also be used for more reliable detection of unimodular eigenvalues of symplectic pencils
like \eqref{eq:symp_pen}.

\subsection{Adapting \cref{alg:owr_bt} for discrete-time $\Kcon$}
\label{sec:certify_disc_fixed}
To create our first discrete-time 2D level-set test, 
we will again look for pairs of points a \emph{fixed distance} $\eta \geq 0$ apart,
but now we will do this along rays from the origin, i.e., 
for $\hat \theta \in [0,2\pi)$, we check if $h(\hat r,\hat \theta) = h(\hat r+\eta, \hat \theta) = \gamma$
holds for some $\hat r > 1$.
Suppose $\gamma$ is a singular value of both $H(r,\theta)$ and $H(r+\eta,\theta)$ 
with respective left and right singular vector pairs $u,v$ and $\hat u,\hat v$.
Then by \cref{thm:1d_circle}, we have that 
\bseq
	\label{eq:eig_pair_radial}
	\begin{align}
	\label{eq:eig_pair_radial1}
	\begin{bmatrix} 
		A    	& \gamma(r-1)I \\
		0 	& rI
	\end{bmatrix}
	\begin{bmatrix} v \\ u \end{bmatrix}
	&= \eit 
	\begin{bmatrix} 
		rI    			& 0 	\\
		\gamma(r-1)I 	& A^*
	\end{bmatrix}
	\begin{bmatrix} v \\ u \end{bmatrix}, 
	\\
	\label{eq:eig_pair_radial2}
	\begin{bmatrix} 
		A    	& \gamma(r+\eta-1)I \\
		0 	& (r+\eta)I
	\end{bmatrix}
	\begin{bmatrix} \hat v \\ \hat u \end{bmatrix}
	&= \eit 
	\begin{bmatrix} 
		(r+\eta)I    			& 0 	\\
		\gamma(r+\eta-1)I 	& A^*
	\end{bmatrix}
	\begin{bmatrix} \hat v \\ \hat u \end{bmatrix},
\end{align}
\eseq
which respectively we denote as $M - \lambda N$ and $\widetilde M - \lambda \widetilde N$.
Now define $W = \begin{bsmallmatrix} v \\ u \end{bsmallmatrix} \begin{bsmallmatrix} \hat v^* & \hat u^* \end{bsmallmatrix} \neq 0$.
Multiplying the two equations above from the right side, respectively by 
$\begin{bsmallmatrix} \hat v^* & \hat u^* \end{bsmallmatrix}\widetilde M^*$
and
$\begin{bsmallmatrix} v^* & u^* \end{bsmallmatrix}N^*$,
yields
\bseq
	\begin{align}
	\label{eq:radial_eig1}
	M W \widetilde M^* &= \eit N W \widetilde M^* \\
	\label{eq:radial_eig2_pre}
	\widetilde M W^* N^* &= \eit \widetilde N W^* N^*.
	\end{align}
\eseq
If we take the conjugate transpose of \eqref{eq:radial_eig2_pre} and then multiply it by $\eit$,
we obtain
\beq
\label{eq:radial_eig2}
	N W \widetilde N^* = \eit N W \widetilde M^*.
\eeq
Subtracting \eqref{eq:radial_eig2} from \eqref{eq:radial_eig1} yields
\beq
	\label{eq:sylv_dr}
	M W \widetilde M^* - N W \widetilde N^* = 0,
\eeq
so as $M - \lambda N$ and $\widetilde N^* - \lambda \widetilde M^*$  are both regular,
\cite[Theorem~1]{Chu87} states that these two pencils must share an eigenvalue
if $W \neq 0$ solves the equation above.
 Moreover if this shared eigenvalue is $\eit$, then $\emit$ is an eigenvalue of 
$\widetilde N - \lambda \widetilde M$, which in turn implies that 
$\eit$ is an eigenvalue of $\widetilde M - \lambda \widetilde N$.
Thus, \eqref{eq:sylv_dr} having a nonzero solution $W$ is a necessary
condition for the two pencils in \eqref{eq:eig_pair_radial} to have eigenvalue $\eit$ in common.

We now want to separate out the $r$ terms of $M W \widetilde M^*$ and $N W \widetilde N^*$.
First, we have:
\beq
	\label{eq:MNtildestar}
	\widetilde M^* =  
	\begin{bmatrix}
		A^*		& 0 			\\
		\gamma(r+\eta-1) I 	& (r+\eta)I 
	\end{bmatrix}
	\quad
	\text{and}
	\quad
	\widetilde N^* =  
	\begin{bmatrix}
		(r+\eta)I	& \gamma(r+\eta-1) I		\\
		0 		& A 
	\end{bmatrix}.
\eeq
Then $MW\widetilde M^*$ is
\beq
	\left( \begin{bmatrix}
		A 	& -\gamma I \\
		0	& 0
	\end{bmatrix}
	+ 
	r
	\begin{bmatrix}
		0 	& \gamma I \\
		0 	& I 		
	\end{bmatrix}
	\right)
	W
	\left(
	\begin{bmatrix}
		A^*				& 0		\\
		\gamma (\eta-1) I 	& \eta I 
	\end{bmatrix}
	+ 
	r
	\begin{bmatrix}
		0 		& 0	\\
		\gamma I 	& I
	\end{bmatrix}
	\right),
\eeq
which is equal to
\begin{multline}
	\begin{bsmallmatrix}
		A 	& -\gamma I \\
		0	& 0
	\end{bsmallmatrix}
	W 
	\begin{bsmallmatrix}
		A^*				& 0		\\
		\gamma (\eta-1) I 	& \eta I 
	\end{bsmallmatrix}
	+ \\
	r \left(
	\begin{bsmallmatrix}
		A 	& -\gamma I \\
		0	& 0
	\end{bsmallmatrix}
	W
	\begin{bsmallmatrix}
		0 		& 0	\\
		\gamma I 	& I
	\end{bsmallmatrix}
	+ 
	\begin{bsmallmatrix}
		0 	& \gamma I \\
		0 	& I 		
	\end{bsmallmatrix}
	W
	\begin{bsmallmatrix}
		A^*				& 0		\\
		\gamma (\eta-1) I 	& \eta I 
	\end{bsmallmatrix}	
	\right)
	+ 
	r^2	
	\begin{bsmallmatrix}
		0 	& \gamma I \\
		0 	& I 		
	\end{bsmallmatrix}
	W	
	\begin{bsmallmatrix}
		0 		& 0	\\
		\gamma I 	& I
	\end{bsmallmatrix}.
\end{multline}
Vectorizing the above equation, with $w = \vecop(W)$,  yields
\begin{multline}
	\begin{bsmallmatrix}
		\overline A 	& \gamma (\eta-1) I \\
		0			& \eta I				
	\end{bsmallmatrix}
	\otimes
	\begin{bsmallmatrix}
		A 		& -\gamma I	\\
		0		& 0
	\end{bsmallmatrix} 
	w
	+ \\
	r \left(
	\begin{bsmallmatrix}
		0 		& \gamma I \\
		0 		& I
	\end{bsmallmatrix}
	\otimes
	\begin{bsmallmatrix}
		A 		& -\gamma I \\
		0		& 0
	\end{bsmallmatrix}
	+
	\begin{bsmallmatrix}
		\overline A 	& \gamma (\eta-1) I \\
		0			& \eta I	
	\end{bsmallmatrix}
	\otimes
	\begin{bsmallmatrix}
		0 		& \gamma I \\
		0 		& I
	\end{bsmallmatrix}	
	\right)
	w
	+
	r^2	
	\begin{bsmallmatrix}
		0 		& \gamma I \\
		0 		& I
	\end{bsmallmatrix}
	\otimes
	\begin{bsmallmatrix}
		0 		& \gamma I \\
		0 		& I		
	\end{bsmallmatrix}
	w,
\end{multline}
which we will abbreviate as
\beq
	\mathcal{M}_0 w + r \mathcal{M}_1 w + r^2 \mathcal{M}_2 w.
\eeq
Likewise, $NW\widetilde N^*$ is
\beq
	\left(
	\begin{bmatrix}
		0 		& 0 \\
		-\gamma I & A^*		
	\end{bmatrix}
	+ 
	r
	\begin{bmatrix}
		I 		& 0\\
		\gamma I	& 0
	\end{bmatrix}
	\right)
	W
	\left(
	\begin{bmatrix}
		\eta I		& \gamma(\eta-1) I \\
		0 		& A 
	\end{bmatrix}
	+ 
	r
	\begin{bmatrix}
		I		& \gamma I  \\
		0 		& 0
	\end{bmatrix}
	\right),
\eeq
which is equal to
\begin{multline}
	\begin{bsmallmatrix}
		0 		& 0 \\
		-\gamma I & A^*	
	\end{bsmallmatrix}
	W
	\begin{bsmallmatrix}
		\eta I		& \gamma(\eta-1) I \\
		0 		& A 
	\end{bsmallmatrix}
	+ \\
	r
	\left(
	\begin{bsmallmatrix}
		0 		& 0 \\
		-\gamma I & A^*	
	\end{bsmallmatrix}
	W	
	\begin{bsmallmatrix}
		I		& \gamma I  \\
		0 		& 0
	\end{bsmallmatrix}
	+
	\begin{bsmallmatrix}
		I 		& 0\\
		\gamma I	& 0
	\end{bsmallmatrix}
	W
	\begin{bsmallmatrix}
		\eta I		& \gamma(\eta-1) I \\
		0 		& A 
	\end{bsmallmatrix}
	\right)
	+
	r^2
	\begin{bsmallmatrix}
		I 		& 0\\
		\gamma I	& 0
	\end{bsmallmatrix}
	W
	\begin{bsmallmatrix}
		I		& \gamma I  \\
		0 		& 0
	\end{bsmallmatrix}.
\end{multline}
Similarly vectorizing this gives
\begin{multline}
	\begin{bsmallmatrix}
		\eta I				& 0 \\
		\gamma(\eta-1) I 	& A\tp 
	\end{bsmallmatrix}
	\otimes
	\begin{bsmallmatrix}
		0 		& 0 \\
		-\gamma I & A^*	
	\end{bsmallmatrix}
	w
	+ \\
	r
	\left(
	\begin{bsmallmatrix}
		I		& 0 \\
		\gamma I 	& 0
	\end{bsmallmatrix}
	\otimes
	\begin{bsmallmatrix}
		0 		& 0 \\
		-\gamma I & A^*	
	\end{bsmallmatrix}
	+
	\begin{bsmallmatrix}
		\eta I				& 0 \\
		\gamma(\eta-1) I 	& A\tp 
	\end{bsmallmatrix}
	\otimes
	\begin{bsmallmatrix}
		I 		& 0\\
		\gamma I 	& 0 
	\end{bsmallmatrix}
	\right)
	w
	+
	r^2
	\begin{bsmallmatrix}
		I 		& 0\\
		\gamma I	& 0 
	\end{bsmallmatrix}
	\otimes
	\begin{bsmallmatrix}
		I 		& 0\\
		\gamma I	& 0
	\end{bsmallmatrix}
	w,
\end{multline}
which we will abbreviate as
\beq
	\mathcal{N}_0 w + r \mathcal{N}_1 w + r^2 \mathcal{N}_2 w.
\eeq
Therefore, we finally have the following quadratic eigenvalue problem:
\beq
	\left( \mathcal{M}_0 - \mathcal{N}_0 \right) w 
	+ r \left( \mathcal{M}_1 - \mathcal{N}_1 \right) w
	+ r^2 \left( \mathcal{M}_2 - \mathcal{N}_2 \right) w = 0,
\eeq
which we will abbreviate as:
\beq
	\label{eq:large_eig_dr}
 	\mathcal{Q}_0 w + r \mathcal{Q}_1 w + r^2 \mathcal{Q}_2 w = 0.
\eeq

Thus, to perform our fixed-distance discrete-time $\Kcon$ certificate 
for adapting \cref{alg:owr_bt}, we first compute all the real eigenvalues $r > 1$ of \eqref{eq:large_eig_dr},
which specify a set of concentric circles centered at the origin on which we may find level set points.
Then, for each of these candidate radii, we apply \cref{thm:1d_circle} to see
if we indeed detect any points on a $\gamma$-level set (or lower) of $h(r,\theta)$.
Like our continuous-time algorithms, this procedure is also $\bigO(\n^6)$ work (albeit with a larger constant term)
when using dense (quadratic) eigensolvers.

Note that for a generic quadratic eigenvalue problem of the form \eqref{eq:large_eig_dr},
if either $\mathcal{Q}_0$ or $\mathcal{Q}_2$ are nonsingular,
then the problem is well posed,
i.e., it has at least one solution and not infinitely many; 
see \cite[p.~283]{BaiDDetal00}.
While for our quadratic eigenvalue problem $\mathcal{Q}_2$ is singular (see \cref{lem:q2_sing} in \cref{apdx:lemmas}), we now show that $\mathcal{Q}_0$ is generically nonsingular, 
and thus \eqref{eq:large_eig_dr} is guaranteed to be well posed.

\begin{theorem}
\label{thm:Q0_nonsing}
Let $\gamma,\eta \in \R$ both be positive.
Then $\mathcal{Q}_0$ from \eqref{eq:large_eig_dr} is nonsingular if and only if
$A$ is nonsingular and $\gamma$ is not a singular value of $A$.
\end{theorem}
\begin{proof}
For any vector $w \in \C^{4\n^2}$, suppose $\mathcal{Q}_0 w = 0$, i.e.,
\[
	\left(
	\begin{bmatrix}
		\, \overline A 	& \gamma (\eta-1) I \\
		0			& \eta I				
	\end{bmatrix}
	\otimes
	\begin{bmatrix}
		A 		& -\gamma I	\\
		0		& 0
	\end{bmatrix} 
	\right)
	w
	-
	\left(
	\begin{bmatrix}
		\eta I				& 0 \\
		\gamma(\eta-1) I 	& A\tp 
	\end{bmatrix}
	\otimes
	\begin{bmatrix}
		0 		& 0 \\
		-\gamma I & A^*	
	\end{bmatrix}
	\right)	
	w 
	= 0,
\]
which holds if and only if
\beq
	\label{eq:gensyl_Q0}
	\begin{bmatrix}
		A 	& -\gamma I \\
		0	& 0
	\end{bmatrix}
	W 
	\begin{bmatrix}
		A^*				& 0		\\
		\gamma (\eta-1) I 	& \eta I 
	\end{bmatrix}
	- 
	\begin{bmatrix}
		0 		& 0 \\
		-\gamma I & A^*	
	\end{bmatrix}
	W
	\begin{bmatrix}
		\eta I		& \gamma(\eta-1) I \\
		0 		& A 
	\end{bmatrix}
	= 0,
\eeq
where $w = \vecop(W)$.
By \cite[Theorem~1]{Chu87}, the generalized Sylvester equation above has a unique solution
if and only if the two matrix pencils
\[
	\begin{bmatrix}
		A 	& -\gamma I \\
		0	& 0
	\end{bmatrix}
	- \lambda 
	\begin{bmatrix}
		0 		& 0 \\
		-\gamma I & A^*	
	\end{bmatrix}
	\qquad \text{and} \qquad
	\begin{bmatrix}
		\eta I		& \gamma(\eta-1) I \\
		0 		& A 
	\end{bmatrix}
	- \lambda
	\begin{bmatrix}
		A^*				& 0		\\
		\gamma (\eta-1) I 	& \eta I 
	\end{bmatrix}
\]
are both regular and have no eigenvalues in common.
Clearly $W=0$ satisfies \eqref{eq:gensyl_Q0},
so as long as zero is the only solution, 
$\mathcal{Q}_0$ is nonsingular for the assumptions.
To prove the forward direction,
we thus show that these three conditions hold.

We begin with the first pencil, which is regular if for at least one value of $\lambda$, the resulting matrix is nonsingular.
Using $\lambda = -1$ results in 
$	
	\big[
	\begin{smallmatrix}
		A 		& -\gamma I \\
		-\gamma I & A^*
	\end{smallmatrix}
	\big]
$.
Since $A$ is invertible, this matrix is invertible if and only if its Schur complement with respect to $A$ is, i.e.,
\[
	0 \neq \det \left( A^*  - \gamma^2 A^{-1} \right) 
	= \det \left( AA^*  - \gamma^2 I \right)/\det(A),
\]
which by definition of singular values, holds if and only if $\gamma$ is not a singular value of $A$.
Thus, the first pencil is regular.
For the second pencil, choosing $\lambda = 0$ results in matrix 
$
	\begin{bsmallmatrix}
		\eta I		& \gamma(\eta-1) I \\
		0 		& A 
	\end{bsmallmatrix}
$,	
which is invertible since $\eta \neq 0$ and $A$ is invertible.
Thus the second pencil is also regular.  Furthermore, this argument also establishes that 
zero cannot be an eigenvalue of the second pencil.

We now show that the two pencils do not have any eigenvalues in common.  We begin by noting that
some of the eigenvalues of the first pencil are infinity, while all the eigenvalues of the second pencil 
are finite.
Let $\lambda \neq 0$ be a finite eigenvalue of the first pencil
with eigenvector $\begin{bsmallmatrix} v \\ u \end{bsmallmatrix}	\neq 0$, hence 
\beq
	\label{eq:pen1_mat_Q0}
	\begin{bmatrix}
		A 				& -\gamma I \\
		\lambda\gamma I	& -\lambda A^*
	\end{bmatrix}
	\begin{bmatrix}
		v \\ u 
	\end{bmatrix}	
	= 0
	\qquad 
	\Longleftrightarrow
	\qquad
	\begin{aligned}
	Av &= \gamma u \\
	A^*u &= \gamma v.
	\end{aligned}
\eeq
Since $\gamma \neq 0$ and $A$ is invertible, $u = 0$ if and only if $v = 0$, 
hence neither are zero.  This means they can be rescaled to each have unit norm and so 
$\gamma$ is a singular value of $A$, a contradiction, hence $\lambda = 0$ must hold.
As zero cannot be an eigenvalue of the second pencil, this part of the proof is complete.

For the reverse direction, first suppose that $A$ is singular.  
Then zero must be an eigenvalue of the second pencil, hence the two pencils
share zero as an eigenvalue and so $\mathcal{Q}_0$ is singular.
Now suppose that $\gamma$ is a singular value of $A$ with left and right singular vectors $\tilde u$ and $\tilde v$
and again consider \eqref{eq:pen1_mat_Q0}.
Since $\begin{bsmallmatrix} v \\ u\end{bsmallmatrix} \coloneqq \begin{bsmallmatrix} \tilde v \\ \tilde u\end{bsmallmatrix} \neq 0$ 
is in the nullspace of the matrix given in \eqref{eq:pen1_mat_Q0} for any $\lambda \in \C$, the first pencil matrix is not regular
and so $\mathcal{Q}_0$ is singular.
\end{proof}

\begin{remark}
It is also possible to derive a globality certificate based on arcs instead of radial segments, where 
for $\hat r > 1$, angle(s) $\hat \theta$ satisfying $h(\hat r,\hat \theta) = h(\hat r,\hat \theta + \eta) = \gamma$ are sought.  We also considered this, but it resulted in a quadratic eigenvalue problem
where both of the corresponding $\mathcal{Q}_0$ and $\mathcal{Q}_2$ matrices were singular,
and so it was unclear if this alternative quadratic eigenvalue problem was well posed or not.
\end{remark}

\subsection{Adapting \cref{alg:owr,alg:tri} for discrete-time $\Kcon$}
\label{sec:disc_adapt_algs_vari}
Now following the variable-distance certificate idea from \S\ref{sec:alternatives}, 
we derive a new discrete-time version to extend \cref{alg:owr,alg:tri} so that they can compute discrete-time Kreiss constants.

\begin{theorem}
\label{thm:kinv_dr_vari}
For $A \in \C^{\n \times \n}$ with $\rho(A) < 1$,
let $\gamma \in [0,1)$,
$\eta \geq 0$, and $(r_\star,\theta_\star)$ be a global minimizer of \eqref{eq:kinv_disc}.
If $\Kinv \leq \gamma$ and $\eta \in [0,2(\gamma - \Kinv)]$, then 
there exists an $r > 1$ and $\theta \in [0,2\pi)$ such that
\beq
	\label{eq:kinv_dr_vari}
	h(r,\theta) = h(\beta r + \delta,\theta) = \gamma,
\eeq
where $\beta \coloneqq 1 - \delta$ and $\delta \coloneqq -\tfrac{\eta}{1 + \gamma}$.
\end{theorem}

\begin{corollary}
\label{cor:kinv_dr_vari}
For $A \in \C^{\n \times \n}$ with $\rho(A) < 1$, let $\gamma \in [0,1)$, $\eta \geq 0$,
and $(r_\star,\theta_\star)$ be a global minimizer of \eqref{eq:kinv_disc}.
If there do not exist any pairs $r,\theta \in \R$ with $r > 1$ such that \eqref{eq:kinv_dr_vari} holds,
then 
\beq
	\label{eq:kinv_dr_vari_lb}
	\Kinv > \gamma - \tfrac{\eta}{2x_\star}.
\eeq
\end{corollary}

\begin{proof}[Proof of \cref{thm:kinv_dr_vari}]
The proof follows similarly to the proof of \cref{thm:kinv_dr}, except that
\eqref{eq:kinv_dr_vari} corresponds to a distance of  
$\tilde \eta \coloneqq (\beta r + \delta) - r = \tfrac{\eta (r - 1)}{1 + \gamma}$
between the two level-set points,
which leads to cancellation, and so $\eta \in [0,2(\gamma - \Kinv)]$.
\end{proof}

For brevity, we skip showing the lengthy derivation of the resulting discrete-time variable-distance certificate
(it follows similarly to \S\ref{sec:certify_disc_fixed}),
and instead just give the key parts necessary to perform the computation.
Suppose $\gamma$ is both a singular value of $H(r,\theta)$ and $H(\beta r + \delta,y)$ 
with respective left and right singular vectors pairs $u$,$v$ and $\hat u$,$\hat v$.
Applying \cref{thm:1d_circle} to $H(r,\theta)$ and $H(\beta r + \delta,\theta)$ yields 
\eqref{eq:eig_pair_radial1} and 	\eqref{eq:eig_pair_radial2} but now with $r$ replaced by $\beta r + \delta$.
For this modified pair of sympletic eigenvalue problems, the certificate derivation results in the large quadratic eigenvalue problem
\begin{align}
	\label{eq:large_eig_dr_vari}
  	&\mathcal{\widetilde Q}_0 w + r \mathcal{\widetilde Q}_1 w + r^2 \mathcal{\widetilde Q}_2 w = 0, 
	\quad \text{where} \\
	\nonumber
	& \mathcal{\widetilde Q}_0 = \mathcal{Q}_0 \ \text{given in \eqref{eq:large_eig_dr} but with $\eta$ replaced by $\delta$}, \\
	\nonumber
	\begin{split}
	&\mathcal{\widetilde Q}_1 = 
	\beta
	\left(
	\begin{bsmallmatrix}
		0 		& \gamma I \\
		0 		& I
	\end{bsmallmatrix}
	\otimes
	\begin{bsmallmatrix}
		A 		& -\gamma I \\
		0		& 0
	\end{bsmallmatrix}
	-
	\begin{bsmallmatrix}
		I		& 0 \\
		\gamma I 	& 0
	\end{bsmallmatrix}
	\otimes
	\begin{bsmallmatrix}
		0 		& 0 \\
		-\gamma I & A^*	
	\end{bsmallmatrix}	
	\right) \\
	& \qquad \qquad \qquad \qquad  +
	\left(
	\begin{bsmallmatrix}
		\overline A 	& \gamma (\delta-1) I \\
		0			& \delta I	
	\end{bsmallmatrix}
	\otimes
	\begin{bsmallmatrix}
		0 		& \gamma I \\
		0 		& I
	\end{bsmallmatrix}	
	-
	\begin{bsmallmatrix}
		\delta I			& 0 \\
		\gamma(\delta-1) I 	& A\tp 
	\end{bsmallmatrix}
	\otimes
	\begin{bsmallmatrix}
		I 		& 0\\
		\gamma I 	& 0 
	\end{bsmallmatrix}
	\right),
	\end{split}
	\\
	\nonumber
	&\mathcal{\widetilde Q}_2 = \beta \mathcal{Q}_2 \ \text{given in \eqref{eq:large_eig_dr}}.
\end{align}
Although $\mathcal{\widetilde Q}_2$ is singular, $\mathcal{\widetilde Q}_0$ is nonsingular
for the assumptions of \cref{thm:Q0_nonsing}, since $\eta \coloneqq \delta \neq 0$.
Thus, \eqref{eq:large_eig_dr_vari} is well posed under the same assumptions.

\subsection{Adapting divide-and-conquer for discrete-time $\Kcon$}
\label{sec:fast_eig_disc}
We now show how the real eigenvalues of \eqref{eq:large_eig_dr} and \eqref{eq:large_eig_dr_vari}
may be computed using divide-and-conquer.
We begin with \eqref{eq:large_eig_dr} and form its (companion) linearization
\beq
	\label{eq:linearization}
	\begin{bmatrix} 
		\mathcal{Q}_1 	& \mathcal{Q}_0 \\ 
		-I 	& 0
	\end{bmatrix}
	z
	= 
	r
	\begin{bmatrix}
		-\mathcal{Q}_2 & 0 \\
		0 	& - I
	\end{bmatrix}
	z,
\eeq
where $z = \begin{bsmallmatrix} rw\\ w \end{bsmallmatrix}$.
Assuming an a priori upper bound $D > 0$ is known for 
all real eigenvalues of \eqref{eq:large_eig_dr},
divide-and-conquer will sweep the interval $[1,D]$ to 
find all the real-valued eigenvalues in this range.
We now detail how the necessary operations with 
the matrices in \eqref{eq:linearization} can all be done in at most $\bigO(\n^3)$ work.

Consider doing matrix-vector products with either matrix in \eqref{eq:linearization} 
and a vector $w = \begin{bsmallmatrix} w_1\\ w_2 \end{bsmallmatrix}$,
where $w_1,w_2 \in \C^{2n^2}$.
The nontrivial parts of these products are:
\bseq
\begin{align}
	\mathcal{Q}_0 w_2 &= \mathcal{M}_0 w_2 - \mathcal{N}_0 w_2 \\
	\mathcal{Q}_1 w_1 &= \mathcal{M}_1 w_1 - \mathcal{N}_1 w_1 \\
	\mathcal{Q}_2 w_1 &= \mathcal{M}_2 w_1 - \mathcal{N}_2 w_1,
\end{align}
\eseq
which are equal to the following respective vectorizations:
\bseq
\label{eq:fast_Qmats}
\begin{align}
	\label{eq:fast_Q0}
	\mathcal{Q}_0 w_2 &=
	\vecop
	\left(
	\begin{bsmallmatrix}
		A 	& -\gamma I \\
		0	& 0
	\end{bsmallmatrix}
	W_2
	\begin{bsmallmatrix}
		A^*				& 0		\\
		\gamma (\eta-1) I 	& \eta I 
	\end{bsmallmatrix}
	-
	\begin{bsmallmatrix}
		0 		& 0 \\
		-\gamma I & A^*	
	\end{bsmallmatrix}
	W_2
	\begin{bsmallmatrix}
		\eta I		& \gamma(\eta-1) I \\
		0 		& A 
	\end{bsmallmatrix}
	\right) \\
	\begin{split}
	\mathcal{Q}_1 w_1 &= 
	\vecop
	\Big(
	\begin{bsmallmatrix}
		A 	& -\gamma I \\
		0	& 0
	\end{bsmallmatrix}
	W_1
	\begin{bsmallmatrix}
		0 		& 0	\\
		\gamma I 	& I
	\end{bsmallmatrix}
	+ 
	\begin{bsmallmatrix}
		0 	& \gamma I \\
		0 	& I 		
	\end{bsmallmatrix}
	W_1
	\begin{bsmallmatrix}
		A^*				& 0		\\
		\gamma (\eta-1) I 	& \eta I 
	\end{bsmallmatrix}
	\\ &{} \qquad \qquad \qquad
	- 
	\begin{bsmallmatrix}
		0 		& 0 \\
		-\gamma I & A^*	
	\end{bsmallmatrix}
	W_1	
	\begin{bsmallmatrix}
		I		& \gamma I  \\
		0 		& 0
	\end{bsmallmatrix}
	-
	\begin{bsmallmatrix}
		I 		& 0\\
		\gamma I	& 0
	\end{bsmallmatrix}
	W_1
	\begin{bsmallmatrix}
		\eta I		& \gamma(\eta-1) I \\
		0 		& A 
	\end{bsmallmatrix}
	\Big) 
	\end{split}
	\\
	\mathcal{Q}_2 w_1 &=
	\vecop
	\left(
	\begin{bsmallmatrix}
		0 		& I \\
		0 		& \gamma I 
	\end{bsmallmatrix}
	W_1	
	\begin{bsmallmatrix}	
		0 		& 0 \\
		I 		& \gamma I
	\end{bsmallmatrix}
	-
	\begin{bsmallmatrix}
		\gamma I 	& 0\\
		I 		& 0
	\end{bsmallmatrix}
	W_1	
	\begin{bsmallmatrix}
		\gamma I 	& I\\
		0 		& 0
	\end{bsmallmatrix}
	\right),	
\end{align}
\eseq
where $w_1 = \vecop(W_1)$ and $w_2 = \vecop(W_2)$.
The first two of these can be obtained in $\bigO(\n^3)$ work since
they only involve matrix-matrix products with $2\n \times 2\n$ matrices.
The third can be obtained in $\bigO(\n^2)$ work as
the number of nonzero entries in $\mathcal{Q}_2$ is simply $8\n^2$,
hence one should just store $\mathcal{Q}_2$ in a sparse format.
This is also fortunate as when applying shift-and-invert to a generalized eigenvalue problem $Ax = \lambda Bx$, solvers such as \texttt{eigs} in \matlab\ require that $B$ is provided explicitly, even when
the operator $(A - sB)^{-1}$ is given implicitly as a function handle.

Given a shift $\shift \in \C$ and a vector
$y = \begin{bsmallmatrix} y_1\\ y_2 \end{bsmallmatrix}$, 
where $y_1,y_2 \in \C^{2n^2}$, we now focus on
\beq
	w = 
	\left(
	\begin{bmatrix} 
		\mathcal{Q}_1 	& \mathcal{Q}_0 \\ 
		-I 	& 0
	\end{bmatrix}
	- \shift
	\begin{bmatrix}
		-\mathcal{Q}_2 & 0 \\
		0 	& - I
	\end{bmatrix}
	\right)^{-1}
	y,
\eeq
which is equivalent to
\beq
	\label{eq:si_disc}
	\left(
	\begin{bmatrix} 
		\mathcal{Q}_1 	& \mathcal{Q}_0 \\ 
		-I 	& 0
	\end{bmatrix}
	- \shift
	\begin{bmatrix}
		-\mathcal{Q}_2 & 0 \\
		0 	& - I
	\end{bmatrix}
	\right)
	\begin{bmatrix} w_1 \\ w_2 \end{bmatrix}
	=
	\begin{bmatrix} 
		\mathcal{Q}_1 + \shift \mathcal{Q}_2	& \mathcal{Q}_0 \\ 
		-I 	& \shift I
	\end{bmatrix}
	\begin{bmatrix} w_1 \\ w_2 \end{bmatrix}
	= 
	\begin{bmatrix} y_1 \\ y_2 \end{bmatrix}.
\eeq
The bottom block row provides:
\beq
	\label{eq:v2}
	\shift w_2 = y_2 + w_1.
\eeq
Multiplying the top block row of \eqref{eq:si_disc} by $\shift$ and 
then substituting in \eqref{eq:v2}, we get
\beq
	\mathcal{Q}_0 w_1 + \shift \mathcal{Q}_1 w_1 + \shift ^2 \mathcal{Q}_2 w_1 
	= \shift y_1 - \mathcal{Q}_0 y_2 \eqqcolon \hat y.
\eeq
Using \eqref{eq:fast_Qmats} to obtain the four matrix-vector products above,
we can then solve for $w_1$ via solving the following generalized continuous-time algebraic Sylvester equation:
\beq
	\label{eq:gen_sylv}
	M W_1 \widetilde M^{*} - N W_1 \widetilde N^{*} = \widehat Y,
\eeq
where $w_1 = \vecop(W_1)$, $\hat y = \vecop(\widehat Y)$, 
and the matrix pairs $M,N$ and $\widetilde M^*,\widetilde N^*$ are respectively 
given in \eqref{eq:eig_pair_radial1} and \eqref{eq:MNtildestar}, 
all with $r$ replaced with $\shift$.
Per \cite{KagW89}, solving \eqref{eq:gen_sylv} can be done in $\bigO(n^3)$ work.
Finally, $w_2$ is obtained via \eqref{eq:v2}.

For \eqref{eq:large_eig_dr_vari}, only a few minor modifications to the divide-and-conquer 
variant we have just explained for \eqref{eq:large_eig_dr} are necessary.
As $\mathcal{\widetilde Q}_0$ is equal to $\mathcal{Q}_0$ with $\eta \coloneqq \delta$ and 
$\mathcal{\widetilde Q}_2 = \beta \mathcal{Q}_2$, the first and third equations in \eqref{eq:fast_Qmats}
can be used to do the corresponding matrix-vector products.
For $w_1 \in \C^{2\n^2}$, we have 
\beq
	\begin{split}
	\mathcal{\widetilde Q}_1 w_1 &= 
	\vecop
	\Big(
	\beta
	\Big(
	\begin{bsmallmatrix}
		A 		& -\gamma I \\
		0		& 0
	\end{bsmallmatrix}
	W_1	
	\begin{bsmallmatrix}
		0 		& 0 \\
		\gamma I	& I
	\end{bsmallmatrix}
	-
	\begin{bsmallmatrix}
		0 		& 0 \\
		-\gamma I & A^*	
	\end{bsmallmatrix}	
	W_1
	\begin{bsmallmatrix}
		I		& \gamma I  \\
		0		& 0
	\end{bsmallmatrix}
	\Big) 
	\\ &{} \qquad \qquad \qquad + \Big(
	\begin{bsmallmatrix}
		0 		& \gamma I \\
		0 		& I
	\end{bsmallmatrix}	
	W_1	
	\begin{bsmallmatrix}
		A^*				& 0 \\
		\gamma (\delta-1) I	& \delta I	
	\end{bsmallmatrix}
	-
	\begin{bsmallmatrix}
		I 		& 0\\
		\gamma I 	& 0 
	\end{bsmallmatrix}
	W_1
	\begin{bsmallmatrix}
		\eta I		& \gamma(\delta-1) \\
		0  		& A 
	\end{bsmallmatrix}
	\Big) \Big).
	\end{split}
\eeq

\section{Numerical experiments}
\label{sec:experiments}
To validate our methods for computing Kreiss constants, 
we implemented proof-of-concepts in MATLAB.
The supplementary material includes the code, test examples,
and a detailed description of the setup in order to reproduce the experiments in this paper.
We plan to add ``production-ready" implementations of our methods to a future release of 
ROSTAPACK \cite{rostapack}.

\begin{figure}[tb]
\centering
\resizebox*{13cm}{!}{
\includegraphics[trim={0.9cm 0.0cm 1.5cm 0.4cm},clip]{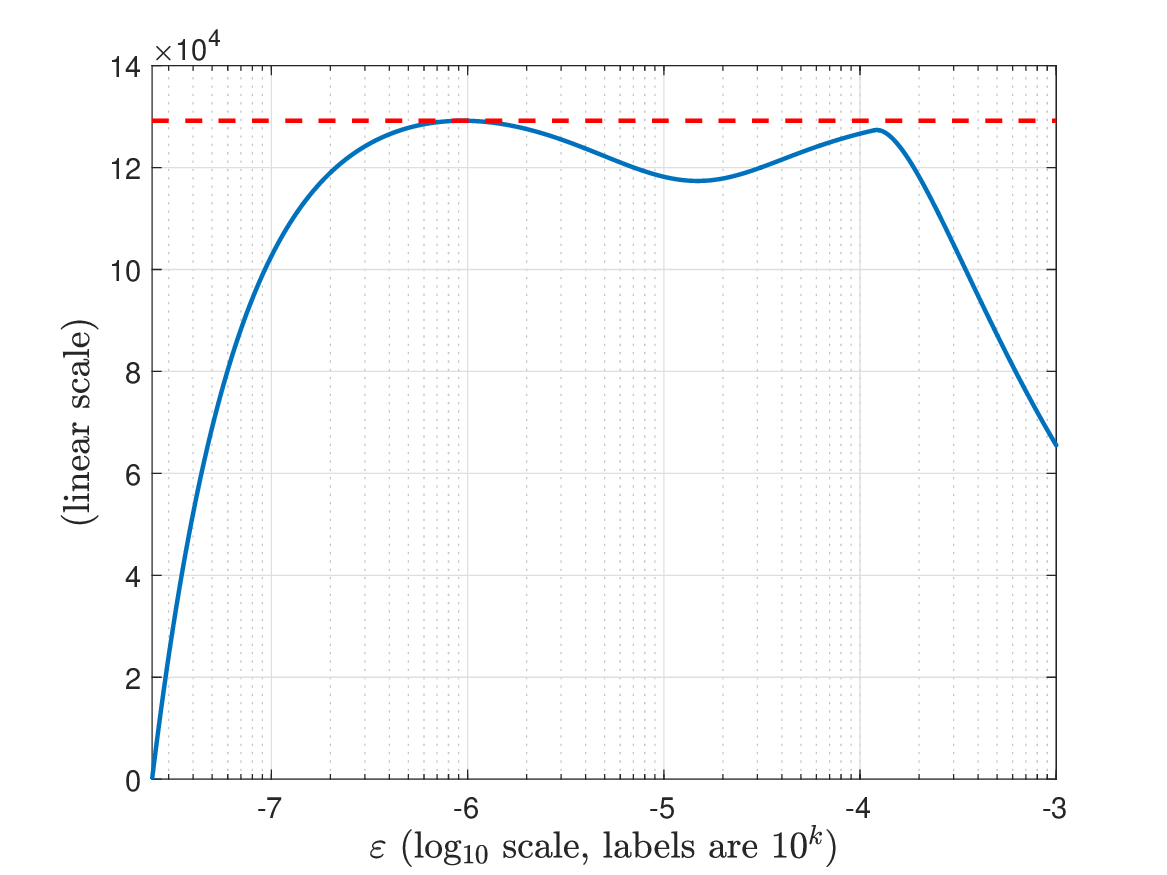}
$\quad$
\includegraphics[trim={0.6cm 0.0cm 1.6cm 0.4cm},clip]{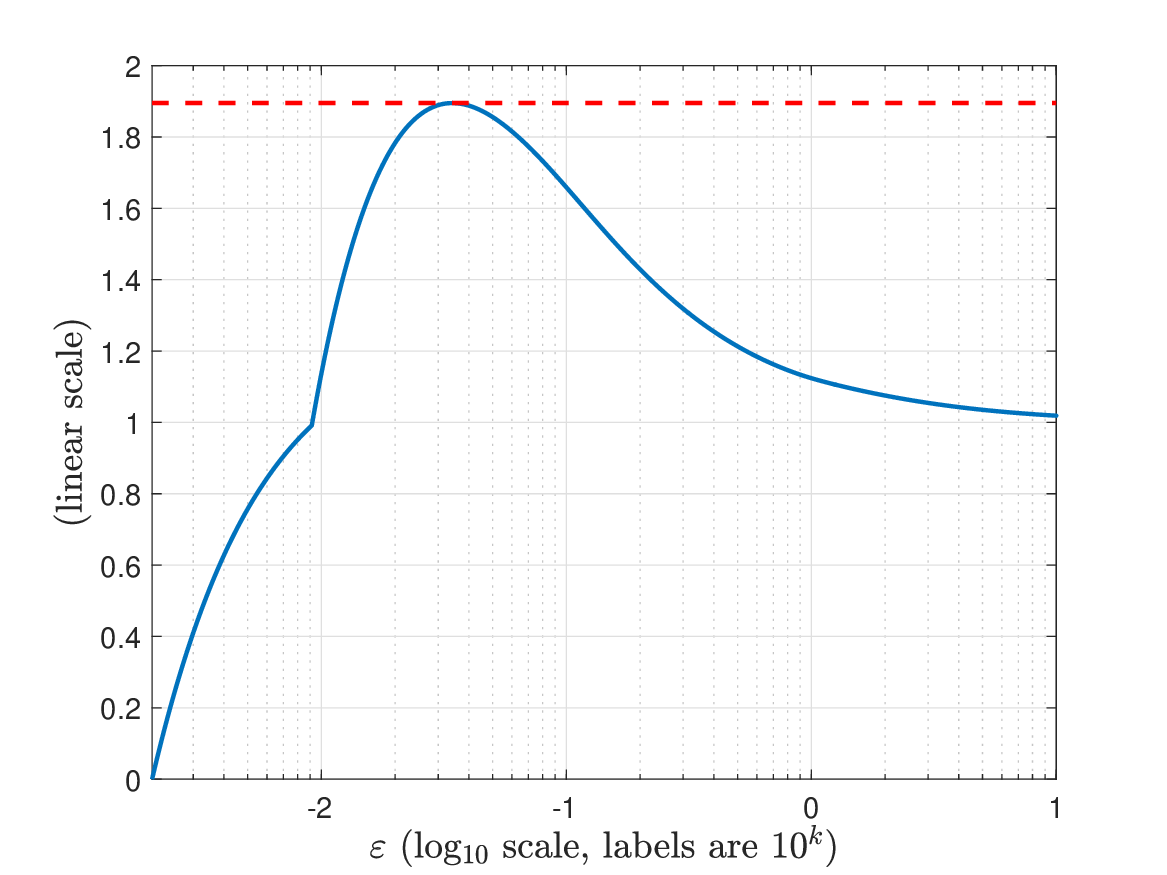}
}
\caption{
The solid curves plot $\tfrac{\aleps(A)}{\eps}$ (left) and $\tfrac{\rhoeps(A)-1}{\eps}$ (right)
as they vary with $\eps$ for our continuous- and discrete-time examples, respectively.
The dashed lines depict the values of $\Kcon$ computed
by our new methods, verifying their global convergence.
The plots were produced using Chebfun \cite{DriHT14}
and the methods of \cite{BenM19}.
}
\label{fig:ratios}
\end{figure}

We considered two $10 \times 10$ stable matrices based on demos in EigTool \cite{eigtool}, 
one for continuous-time $\Kcon$
and a second for discrete-time $\Kcon$.
For the former, we used $A = B - \kappa I$, 
where $B = \texttt{companion\_demo(10)}$ and $\kappa = 1.001\alpha(B)$;
this matrix has a large Kreiss constant, and as shown in \cref{fig:ratios},
\eqref{eq:k1d_cont} has two local maximizers for this example.
For the latter, we chose $A = \tfrac{1}{13}B + \tfrac{11}{10}I$,
where $B = \texttt{convdiff\_demo(11)}$;
while this matrix a small Kreiss constant,
it is interesting for testing as, for this example, 
$h(r,\theta)$ has several local minimizers (see the supplementary material)
and, per \cref{fig:ratios}, \eqref{eq:k1d_disc} appears to be nonsmooth (though not at the maximizer).
To verify restarting, we intentionally chose starting points
such that global minimizers of $g(x,y)$ and $h(r,\theta)$
would not be found in the first round of optimization.
In \cref{table:comp}, we provide detailed metrics on our three algorithms
and see in practice that the optimization-based \cref{alg:owr_bt,alg:owr} are much
faster and more accurate than the trisection-based \cref{alg:tri}.
The much higher numerical accuracy of our two optimization methods is also
verified by comparing to Chebfun, with Chebfun being much slower, 
even though it was given a small interval containing a global maximizer attaining $\Kcon$.
In fact, Chebfun is also much slower than \cref{alg:tri}, which is
the slowest of our three new methods, due to the linear convergence of trisection.
As can be seen, \cref{alg:owr} appears to be the best of the algorithms.
See the supplementary material for additional figures 
showing how \cref{alg:owr_bt,alg:owr} progress one from 
minimizer to the next to converge to $\Kcon$.

We also evaluated our theoretically faster divide-and-conquer approaches.
The supplementary material includes a demo for reproducing one key experiment of the many we performed. 
This demo uses the computations described in \S\ref{sec:fast_eig_cont} and \S\ref{sec:fast_eig_cont_vari} 
to respectively compute eigenvalues of \eqref{eq:large_eig_ctheta} and \eqref{eq:large_eig_cv_vari}
closest to a given shift, where $\gamma$ and $\eta$ were set to the values used in the first 2D level-set test
computed by \cref{alg:owr_bt} when computing $\Kcon$ of our continuous-time example.
However, we observed that these shift-and-invert closest eigenvalue computations were 
somewhat unreliable when compared to using \texttt{eig} in \matlab\ on \eqref{eq:large_eig_ctheta} and \eqref{eq:large_eig_cv_vari} directly.
For space reasons, we forgo many of the details and instead make a few key observations.
First, sometimes the eigenvalues computed by the shift-and-invert approach 
did not have even a single digit of agreement to those computed by \texttt{eig},
but as these eigenvalues were often very close to the origin, this is not so surprising.
Second, we also observed that \texttt{eigs} in \matlab\ would sometimes fail to return one of the closest eigenvalues
to a given shift; interestingly, this phenomenon can even be observed in the 
paper where divide-and-conquer for $\dtu$ was proposed \cite[top left plot of Fig.~3.1]{GuMOetal06}.
Third, we also observed that the large eigenvalue problems 
sometimes have very close conjugate pair eigenvalues, which may explain 
some of the aforementioned difficulties with extracting eigenvalues on the real axis accurately.
As divide-and-conquer is built on the assumption that shift-and-invert eigensolvers
are reliable,
addressing the numerical issues above may just be a matter of picking the right sparse eigensolver.
That being said, our initial testing of divide-and-conquer and its numerical reliability 
is merely a starting point for future work.

\begin{table}[!t]
\centering
\small
\begin{tabular}{ l  c  c  c  c  r } 
\toprule
& $\Kcon$ & $z_0$ & level-set tests & restarts & seconds \\
\midrule
Chebfun            & $1.29186707005845 \times 10^5$       &  ---             &  ---  &  ---  &    97.28\\ 
\cref{alg:owr_bt}  & $1.29186707011257 \times 10^5$       & $6+6\imagunit$   &    14 &     2 &     1.95\\ 
\cref{alg:owr}     & $1.29186707015035 \times 10^5$       & $6+6\imagunit$   &     2 &     2 &     0.76\\ 
\cref{alg:tri}     & $1.29181067678395 \times 10^5$      & $1$              &    89 &  ---  &    21.48\\ 
\midrule
Chebfun            & 1.89501339090609       &  ---             &  ---  &  ---  &    74.04\\ 
\cref{alg:owr_bt}  & 1.89501339090580       & $-1+1\imagunit$  &    15 &     3 &     7.18\\ 
\cref{alg:owr}     & 1.89501339090580       & $-1+1\imagunit$  &     3 &     3 &     1.62\\ 
\cref{alg:tri}     & 1.89501305930067       & $2$              &    81 &  ---  &    41.04\\
\bottomrule
\end{tabular}
\caption{The upper half of the table shows data
for our continuous-time example, while the lower half is for the discrete-time
example. A dash indicates the column is not relevant for the given method.
Chebfun is simply taking the max of the chebfuns produced to make 
\cref{fig:ratios}, where Chebfun was supplied a reasonably small 
interval already known to contain the maximizer attaining $\Kcon$.
}
\label{table:comp}
\end{table}

\section{Concluding remarks}
\label{sec:conclusion}
We have presented the first globally convergent algorithms for computing continuous- and
discrete-time Kreiss constants to arbitrary accuracy.  
\cref{alg:owr}, which uses our variable-distance 2D level-set tests, 
appears to be the fastest of the three different iterations proposed here, since it requires performing the least number of these tests.
However, for continuous-time Kreiss constants, upcoming fast deflation methods \cite[Chapter~3.1]{Koe21} may enable
our fixed-distance 2D level-set test of \S\ref{sec:2d_test_fixed} to be computed with significantly less work, which could
make \cref{alg:owr_bt} more competitive, perhaps even faster than \cref{alg:owr}.
While we have outlined the key theoretical differences between the various large eigenvalue problems
underlying our fixed- and variable-distance certificates,
a systematic numerical evaluation could help illuminate whether there are any notable differences in terms of numerical reliability
when using dense eigensolvers.
Similarly, further investigation into the numerics of our asymptotically faster divide-and-conquer variants using sparse shift-and-invert eigensolvers could prove useful.

\section*{Acknowledgments}
Many thanks to Mark Embree for suggesting to look at the problem 
of computing Kreiss constants and to Michael L. Overton for 
pointing out the possible connection to computing the distance to uncontrollability and hosting the author at the Courant Institute in New York for several visits.
The author is also grateful for the many useful suggestions from the referees, one of whom 
deserves special mention for pointing out both a more concise proof of \cref{lem:foverx} and the simpler explicit inverse formulas for 
$\mathcal{B}_2$ and $\mathcal{\widetilde B}_2$ respectively given by \eqref{eq:B2_inv} and \eqref{eq:B2tilde_inv}.

\bibliographystyle{alpha}
\bibliography{csc,mor,software}

\newcommand{\etalchar}[1]{$^{#1}$}
\begin{thebibliography}{GMO{\etalchar{+}}06}

\bibitem[BDD{\etalchar{+}}00]{BaiDDetal00}
Z.~Bai, J.~Demmel, J.~Dongarra, A.~Ruhe, and H.~{van der Vorst, editors}.
\newblock {\em Templates for the Solution of Algebraic Eigenvalue Problems: A
  Practical Guide}.
\newblock SIAM, Philadelphia, PA, 2000.

\bibitem[BLO03]{BurLO03}
J.~V. Burke, A.~S. Lewis, and M.~L. Overton.
\newblock Robust stability and a criss-cross algorithm for pseudospectra.
\newblock {\em {IMA} J. Numer. Anal.}, 23(3):359--375, 2003.

\bibitem[BLO04]{BurLO04}
J.~V. Burke, A.~S. Lewis, and M.~L. Overton.
\newblock Pseudospectral components and the distance to uncontrollability.
\newblock {\em {SIAM} J. Matrix Anal. Appl.}, 26(2):350--361, 2004.

\bibitem[BM19]{BenM19}
P.~Benner and T.~Mitchell.
\newblock Extended and improved criss-cross algorithms for computing the
  spectral value set abscissa and radius.
\newblock {\em {SIAM} J. Matrix Anal. Appl.}, 40(4):1325--1352, 2019.

\bibitem[BMX98]{BenMX98a}
P.~Benner, V.~Mehrmann, and H.~Xu.
\newblock A numerically stable, structure preserving method for computing the
  eigenvalues of real {H}amiltonian or symplectic pencils.
\newblock {\em Numer. Math.}, 78(3):329--358, 1998.

\bibitem[BMX99]{BenMX98b}
P.~Benner, V.~Mehrmann, and H.~Xu.
\newblock A note on the numerical solution of complex {H}amiltonian and
  skew-{H}amiltonian eigenvalue problems.
\newblock {\em Electron. Trans. Numer. Anal.}, 8:115--126, 1999.

\bibitem[Bur79]{Bur79}
R.~B. Burckel.
\newblock {\em An introduction to classical complex analysis. {V}ol. 1},
  volume~64 of {\em Pure and Applied Mathematics}.
\newblock Birkh\"{a}user Basel, 1979.

\bibitem[BW19]{morBenW19b}
P.~Benner and S.~W.~R. Werner.
\newblock {MORLAB} -- {Model Order Reduction LABoratory} (version 5.0), 2019.
\newblock \url{http://www.mpi-magdeburg.mpg.de/projects/morlab}.

\bibitem[BW20]{morBenW20d}
P.~Benner and S.~W.~R. Werner.
\newblock {H}ankel-norm approximation of large-scale descriptor systems.
\newblock {\em Adv. Comput. Math.}, 46(3):40, 2020.

\bibitem[Chu87]{Chu87}
E.~K.-w. Chu.
\newblock The solution of the matrix equations {$AXB-CXD=E$} and
  {$(YA-DZ,YC-BZ)=(E,F)$}.
\newblock {\em Linear Algebra Appl.}, 93(0):93--105, 1987.

\bibitem[DHT14]{DriHT14}
T.~A Driscoll, N.~Hale, and L.~N. Trefethen.
\newblock {\em Chebfun Guide}.
\newblock Pafnuty Publications, 2014.
\newblock \url{http://www.chebfun.org/docs/guide/}.

\bibitem[Eis84]{Eis84}
R.~Eising.
\newblock Between controllable and uncontrollable.
\newblock {\em Syst. Control Lett.}, 4(5):263--264, 1984.

\bibitem[EK17]{EmbK17}
M.~Embree and B.~Keeler.
\newblock Pseudospectra of matrix pencils for transient analysis of
  differential-algebraic equations.
\newblock {\em {SIAM} J. Matrix Anal. Appl.}, 38(3):1028--1054, 2017.

\bibitem[GMO{\etalchar{+}}06]{GuMOetal06}
M.~Gu, E.~Mengi, M.~L. Overton, J.~Xia, and J.~Zhu.
\newblock Fast methods for estimating the distance to uncontrollability.
\newblock {\em {SIAM} J. Matrix Anal. Appl.}, 28(2):477--502, 2006.

\bibitem[GO11]{GugO11}
N.~Guglielmi and M.~L. Overton.
\newblock Fast algorithms for the approximation of the pseudospectral abscissa
  and pseudospectral radius of a matrix.
\newblock {\em {SIAM} J. Matrix Anal. Appl.}, 32(4):1166--1192, 2011.

\bibitem[Gu00]{Gu00}
M.~Gu.
\newblock New methods for estimating the distance to uncontrollability.
\newblock {\em {SIAM} J. Matrix Anal. Appl.}, 21(3):989--1003, 2000.

\bibitem[K{\"o}h21]{Koe21}
Martin K{\"o}hler.
\newblock {\em Approximate Solution of Non-Symmetric Generalized Eigenvalue
  Problems and Linear Matrix Equations on {HPC} Platforms}.
\newblock {D}issertation, Otto-von-Guericke-Universit{\"a}t, Magdeburg,
  Germany, 2021.
\newblock In Preparation.

\bibitem[Kre62]{Kre62}
H.-O. Kreiss.
\newblock \"{U}ber die {S}tabilit\"{a}tsdefinition f\"{u}r
  {D}ifferenzengleichungen die partielle {D}ifferentialgleichungen
  approximieren.
\newblock {\em {BIT} Numerical Mathematics}, 2(3):153--181, 1962.

\bibitem[KV14]{KreV14}
D.~Kressner and B.~Vandereycken.
\newblock Subspace methods for computing the pseudospectral abscissa and the
  stability radius.
\newblock {\em {SIAM} J. Matrix Anal. Appl.}, 35(1):292--313, 2014.

\bibitem[KW89]{KagW89}
B.~K{\aa}gstr{\"o}m and L.~Westin.
\newblock Generalized {S}chur methods with condition estimators for solving the
  generalized {S}ylvester equation.
\newblock {\em {IEEE} Trans. Autom. Control}, 34(7):745--751, July 1989.

\bibitem[Lan64]{Lan64}
P.~Lancaster.
\newblock On eigenvalues of matrices dependent on a parameter.
\newblock {\em Numer. Math.}, 6:377--387, 1964.

\bibitem[Lar]{propack}
R.~M. Larsen.
\newblock {PROPACK} - {S}oftware for large and sparse {SVD} calculations.
\newblock \url{http://sun.stanford.edu/~rmunk/PROPACK}.

\bibitem[Lau05]{Lau05}
A.~J. Laub.
\newblock {\em Matrix analysis for scientists and engineers}.
\newblock {SIAM} Publications, Philadelphia, PA, 2005.

\bibitem[LT84]{LevT84}
R.~J. LeVeque and L.~N. Trefethen.
\newblock On the resolvent condition in the {K}reiss matrix theorem.
\newblock {\em {BIT} Numerical Mathematics}, 24(4):584--591, 1984.

\bibitem[Men06]{Men06}
E.~Mengi.
\newblock {\em Measures for Robust Stability and Controllability}.
\newblock PhD thesis, New York University, New York, NY 10003, USA, September
  2006.

\bibitem[Mit]{rostapack}
T.~Mitchell.
\newblock {ROSTAPACK}: {RO}bust {STA}bility {PACK}age.
\newblock \url{http://timmitchell.com/software/ROSTAPACK}.

\bibitem[Mit19]{Mit19a}
T.~Mitchell.
\newblock Fast interpolation-based globality certificates for computing
  {K}reiss constants and the distance to uncontrollability.
\newblock e-print arXiv:1910.01069, arXiv, October 2019.
\newblock math.OC.

\bibitem[OW95]{OveW95}
M.~L. Overton and R.~S. Womersley.
\newblock Second derivatives for optimizing eigenvalues of symmetric matrices.
\newblock {\em {SIAM} J. Matrix Anal. Appl.}, 16(3):697--718, 1995.

\bibitem[Spi91]{Spi91}
M.~N. Spijker.
\newblock On a conjecture by {L}e{V}eque and {T}refethen related to the
  {K}reiss matrix theorem.
\newblock {\em {BIT} Numerical Mathematics}, 31(3):551--555, 1991.

\bibitem[TE05]{TreE05}
L.~N. Trefethen and M.~Embree.
\newblock {\em Spectra and pseudospectra: {T}he behavior of nonnormal matrices
  and operators}.
\newblock Princeton University Press, Princeton, NJ, 2005.

\bibitem[Wri02]{eigtool}
T.~G. Wright.
\newblock Eig{T}ool.
\newblock \url{http://www.comlab.ox.ac.uk/pseudospectra/eigtool/}, 2002.

\end{thebibliography}

\appendix

\section{A 2D level-set test for variable-distance pairs with horizontal orientation}
\label{apdx:2d_test_ch_vari}
If we modify \cref{thm:kinv_ch} to consider horizontal pairs of points 
a variable distance $\tfrac{x\eta}{1 + \gamma}$ apart 
(instead a fixed distance $\eta$ apart), to induce cancellation in the proof, 
we obtain the following results.

\begin{theorem}
\label{thm:kinv_ch_vari}
For $A \in \C^{\n \times \n}$ with $\alpha(A) < 0$,
let $\gamma \in [0,1)$,
$\eta \geq 0$, and $(x_\star,y_\star)$ be a global minimizer of \eqref{eq:kinv_cont}.
If $\Kinv \leq \gamma$ and $\eta \in [0,2(\gamma - \Kinv)]$, then 
there exists a pair $x,y \in \R$ with $x > 0$ such that
\beq
	\label{eq:kinv_ch_vari_ub}
	g(x,y) = g(\beta x,y) = \gamma,
\eeq
where $\beta \coloneqq 1 + \tfrac{\eta}{1 + \gamma}$.
\end{theorem}

\begin{corollary}
\label{cor:kinv_ch_vari}
For $A \in \C^{\n \times \n}$ with $\alpha(A) < 0$, let $\gamma \in [0,1)$, $\eta \geq 0$,
and $(x_\star,y_\star)$ be a global minimizer of \eqref{eq:kinv_cont}.
If there do not exist any pairs $x,y \in \R$ with $x > 0$ such that \eqref{eq:kinv_ch_vari_ub} holds,
then 
\beq
	\label{eq:kinv_ch_vari_lb}
	\Kinv > \gamma - \tfrac{\eta}{2x_\star}.
\eeq
\end{corollary}

\begin{proof}[Proof of \cref{thm:kinv_ch_vari}]
The proof follows similarly to the proof of \cref{thm:kinv_ch}, except that
\eqref{eq:kinv_ch_vari_ub} corresponds to a distance of $\tilde \eta \coloneqq \beta x - x = \tfrac{\eta x}{1+\gamma}$ between the two level-set points, which leads to cancellation, and so $\eta \in [0,2(\gamma - \Kinv)]$.
\end{proof}

The derivation of the corresponding verification procedure, to enable 
versions of \cref{alg:owr,alg:tri} using \eqref{eq:kinv_ch_vari_ub}, is as follows.
Suppose $\gamma$ is both a singular value of $G(x,y)$ and $G(\beta x,y)$ 
with respective left and right singular vectors pairs $u$,$v$ and $\hat u$,$\hat v$.
Applying \cref{thm:1d_vert} to $G(x,y)$ and  $G(\beta x,y)$ yields the 
following two Hamiltonian standard eigenvalue problems
\beq
\label{eq:eig_pair_ch_vari}
	\begin{bmatrix}
		A - xI 		& \gamma xI 	\\
		-\gamma xI  	& xI -A^*  		\\
	\end{bmatrix}
	\begin{bmatrix} v \\ u \end{bmatrix}
	= \imagunit y \begin{bmatrix} v \\ u \end{bmatrix} 
	\ \text{and} \
	\begin{bmatrix}
		A - \beta xI 		& \gamma \beta x I	 	\\
		-\gamma \beta x I 	& \beta x I - A^*  \\
	\end{bmatrix}
	\begin{bmatrix} \hat v \\ \hat u \end{bmatrix}
	= \imagunit y \begin{bmatrix} \hat v \\ \hat u \end{bmatrix},
\eeq
which have a common eigenvalue if
\beq
	\label{eq:sylv_ch_vari}
	\begin{bmatrix}
		A - xI 		& \gamma xI 	\\
		-\gamma xI  	& xI -A^*  		\\
	\end{bmatrix}
	W + W
	\begin{bmatrix}
		A^* - \beta xI 		& -\gamma \beta x I	\\
		\gamma \beta x I 	& \beta x I - A  		\\
	\end{bmatrix}
	= 0
\eeq
has a nonzero solution $W \in \C^{2\n\times2\n}$.
Separating this into two Sylvester forms to isolate $x$ and then vectorizing
yields the generalized eigenvalue problem
\beq
	\label{eq:large_eig_ch_vari}
	\mathcal{B}_1 w = x \mathcal{\widetilde B}_2 w, 
	\quad \text{where} \quad
	\mathcal{\widetilde B}_2 = 
	I_{2n} \otimes 
	\begin{bmatrix} 
		I 			& -\gamma I \\ 
		\gamma I  	& -I 
	\end{bmatrix} 
	+
	\beta
	\begin{bmatrix} 
		I 		& -\gamma I \\ 
		\gamma I  & -I 
	\end{bmatrix} 
	\otimes I_{2n},
\eeq
$\mathcal{B}_1$ is defined in \eqref{eq:large_eig_cv_vari}, and  $w = \vecop(W)$.
Applying \cref{thm:kronsum} to $\mathcal{\widetilde B}_2$, we see that
its eigenvalues are the pairwise sums of $\pm \sqrt{1 - \gamma^2}$ and 
$\pm \beta \sqrt{1 - \gamma^2}$, hence it is nonsingular if and only if 
$\beta \neq \pm 1$ and $\gamma \neq \pm 1$.
Furthermore, via \cref{lem:b2tilde_inverse}, 
\beq
	\label{eq:B2tilde_inv}
	\mathcal{\widetilde B}_2^{-1} = 	2\omega\phi \mathcal{\widetilde B}_2 - \omega \mathcal{\widetilde B}_2^3,
\eeq
where $\phi \coloneqq (1 + \beta^2)(1 - \gamma^2)$ and $\omega \coloneqq (1-\beta^2)^{-2}(1 - \gamma^2)^{-2}$.
However, as $\eta \to 0$, $\beta \to 1$, and so $\mathcal{\widetilde B}_2$
will become closer and closer to being singular as the algorithms progress.
This suggests that it would be better numerically to solve \eqref{eq:large_eig_ch_vari}
as a generalized eigenvalue problem, rather than solving the standard eigenvalue problem 
$\mathcal{\widetilde B}_2^{-1}\mathcal{B}_1$.
In terms of the number of arithmetic operations, 
if $A$ is real, using divide-and-conquer to compute the positive real eigenvalues of 
\eqref{eq:large_eig_ch_vari} only involves matrix problems with real coefficients.
The two eigenvalue problems in \eqref{eq:eig_pair_ch_vari} also have real coefficients when $A$ is real.
Given $s \in \R$, $y = (\mathcal{B}_1 - s\mathcal{\widetilde B}_2)^{-1}w$ can be computed 
by solving the Sylvester equation \eqref{eq:sylv_ch_vari} 
with $x$ replaced by $s$ and its zero right hand side replaced by $Y$, 
where $w = \vecop(W)$ and $y = \vecop(Y)$.
In contrast, the computations for the 
divide-and-conquer variant from \S\ref{sec:fast_eig_cont_vari} requires complex arithmetic.

\section{Some technical results}
\label{apdx:lemmas}

\begin{lemma}
\label{lem:tri_eta}
Let $\{\gamma_j\}$ be the iterates produced by trisection with $\{\gamma_j\} \to \gamma_\star$,
and suppose that trisection terminates at the $k$th iterate with $| \gamma_\star - \gamma_k | \leq \psi \gamma_\star$ 
holding for some given relative error tolerance $\psi > 0$.  
Then $\eta_k \leq  (1+\psi) \gamma_\star$.
\end{lemma}
\begin{proof}
By construction,
$\gamma_k = L + \eta_k$ and $L \geq 0$ so $\eta_k \leq \gamma_k$ always holds.
If $\gamma_k \geq \gamma_\star$, then
$\psi \gamma_\star \geq | \gamma_\star - \gamma_k | = \gamma_k - \gamma_\star$, which implies
\mbox{$\eta_k \leq  (1+\psi) \gamma_\star$}.
Otherwise, $\gamma_k < \gamma_\star$ must hold, and so it follows that 
\mbox{$\psi \gamma_\star \geq | \gamma_\star - \gamma_k | = \gamma_\star - \gamma_k > \gamma_k - \gamma_\star$}.
\end{proof}

\begin{corollary}
\label{cor:rel_acc}
Given a fixed $\psi > 0$, if 
$\eta_k > (1+\psi) \gamma_\star$, then the relative error of trisection at the $k$th iterate is bounded below by $\psi$, 
specifically $| \gamma_\star - \gamma_k | > \psi \gamma_\star$.  
\end{corollary}

\begin{lemma}
\label{lem:b2tilde_inverse}
Let $C \coloneqq \begin{bsmallmatrix} aI & -bI \\ bI & -aI\end{bsmallmatrix} \in \C^{2n \times 2n}$ with $a,b \in \C$
and let $\beta \in \C$.
Then the matrix $\mathcal{D} \coloneqq I_{2n} \otimes C + \beta C \otimes I_{2n} $
is invertible if and only if $a^2 \neq b^2$ and $\beta^2 \neq 1$, where 
\beq
	\mathcal{D}^{-1} = 
	2\omega\phi D - \omega \mathcal{D}^3 ,
	\eeq
$\phi \coloneqq (1 + \beta^2)(a^2 - b^2)$, and $\omega \coloneqq (1-\beta^2)^{-2}(a^2 - b^2)^{-2}$.
\end{lemma}
\begin{proof}
First note that $C^2 = (a^2 - b^2)I_{2n}$, while using the mixed-product property of $\otimes$ yields
\begin{align*}
	\mathcal{D}^2 
	&= (I_{2n} \otimes C^2 + \beta^2 C^2 \otimes I_{2n}) + \beta (C \otimes C + C \otimes C) \\
	&= (1+\beta^2)(a^2 - b^2)I_{4n^2} + 2\beta C\otimes C = \phi I_{4n^2} + 2\beta C\otimes C.
\end{align*}
Using this equivalence for $\mathcal{D}^2$, the formula for $\mathcal{D}^{-1}$ is verified via 
\begin{align*}
	\mathcal{D}^{-1}\mathcal{D} 
	= \omega ( 2\phi I_{4n^2} - \mathcal{D}^2 ) \mathcal{D}^2 
	&= \omega ( 2\phi I_{4n^2} - (\phi I_{4n^2} + 2\beta C\otimes C) ) (\phi I_{4n^2} + 2\beta C\otimes C) \\
	&= \omega ( \phi I_{4n^2} - 2 \beta C\otimes C ) (\phi I_{4n^2} + 2\beta C\otimes C) \\
	&= \omega (\phi^2 I_{4n^2} - 4 \beta^2 C^2 \otimes C^2) \\
	&= \omega (\phi^2 - 4\beta^2(a^2 - b^2)^2) I_{4n^2},
\end{align*}
which, by noting $\phi^2 - 4\beta^2 (a^2 - b^2)^2 = \omega^{-1}$, is equivalent to $I_{4n^2}$.
\end{proof}

\begin{lemma}
\label{lem:q2_sing}
Matrix $\mathcal{Q}_2$ from \eqref{eq:large_eig_dr} is singular.
\end{lemma}
\begin{proof}
Suppose that $\mathcal{Q}_2 w = 0$ if and only if $w = 0$, and so 
\[
	\begin{bsmallmatrix}
		0 		& \gamma I \\
		0 		& I
	\end{bsmallmatrix}
	\otimes
	\begin{bsmallmatrix}
		0 		& \gamma I \\
		0 		& I		
	\end{bsmallmatrix}
	w
	-
	\begin{bsmallmatrix}
		I 		& 0\\
		\gamma I	& 0 
	\end{bsmallmatrix}
	\otimes
	\begin{bsmallmatrix}
		I 		& 0\\
		\gamma I	& 0
	\end{bsmallmatrix}
	w 
	= 0
	\ \Leftrightarrow \
	\begin{bsmallmatrix}
		0 		& \gamma I \\
		0 		& I 
	\end{bsmallmatrix}
	W	
	\begin{bsmallmatrix}	
		0 		& 0 \\
		\gamma I 	& I
	\end{bsmallmatrix}
	-
	\begin{bsmallmatrix}
		I 		& 0\\
		\gamma I 	& 0
	\end{bsmallmatrix}
	W
	\begin{bsmallmatrix}
		I 		& \gamma I \\
		0 		& 0
	\end{bsmallmatrix}
	= 0,
\]
where the equivalence holds by unvectorizing the first equation
and $w = \vecop(W)$.
This generalized Sylvester equation has a unique solution
if and only if 
\[
	\begin{bsmallmatrix}
		0 	& \gamma I \\
		0 	& I 		
	\end{bsmallmatrix}
	- \lambda 
	\begin{bsmallmatrix}
		I 		& 0\\
		\gamma I	& 0 
	\end{bsmallmatrix}
	\qquad \text{and} \qquad
	\begin{bsmallmatrix}
		I		& \gamma I  \\
		0 		& 0
	\end{bsmallmatrix}
	- \lambda
	\begin{bsmallmatrix}
		0 		& 0	\\
		\gamma I 	& I
	\end{bsmallmatrix}
\]
are both regular matrix pencils and have no eigenvalues in common.
However, as 
$
	\begin{bsmallmatrix}
		0 	& \gamma I \\
		0 	& I 		
	\end{bsmallmatrix}
$
and 
$
	\begin{bsmallmatrix}
		I		& \gamma I  \\
		0 		& 0
	\end{bsmallmatrix}
$
are both singular, the two pencils share zero as an eigenvalue.
\end{proof}

\end{document}